\documentclass[12pt]{article}
\parskip=\medskipamount
\usepackage{amsmath,amsthm,amssymb}
\usepackage{graphicx,epsfig}
\usepackage{exscale}

\numberwithin{equation}{section}

\newtheorem{theorem}{Theorem}

\def\be{\begin{equation}}
\def\ee{\end{equation}}
\def\ve{\varepsilon}
\def\vp{\varphi}
\def\arrowk{^\to{\kern -6pt\topsmash k}}
\def\arrowK{^{^\to}{\kern -9pt\topsmash K}}
\def\arrowt{^\to{\kern -6pt\topsmash t}}
\def\arrowr{^\to{\kern-6pt\topsmash r}}
\def\arrowvp{^\to{\kern -8pt\topsmash\vp}}
\def\tk{\tilde{\kern 1 pt\topsmash k}}
\def\barm{\bar{\kern-.2pt\bar m}}
\def\barN{\bar{\kern-1pt\bar N}}
\def\barA{\, \bar{\kern-3pt \bar A}}

\def\be{\begin{equation}}
\def\ee{\end{equation}}
\usepackage{amsmath,amssymb,amsbsy,amsfonts,amsthm,latexsym,
                amsopn,amstext,amsxtra,euscript,amscd}

\def\ve{\varepsilon}

\def\iint{\not\kern -4pt\int}
\def\iiint{{\small{_\not}}\kern-3.5pt\int}
\def\mod{\text{\rm mod }}
\def\be{\begin{equation}}
\def\ee{\end{equation}}
\numberwithin{equation}{section}

\begin{document}
\title{Short character sums for composite moduli \footnote{2000 {\it Mathematics Subject
Classification}. 11L40, 11M06.}
\footnote{{\it Key words}. character sums, zero free regions}}
\author{Mei-Chu Chang\footnote{Research partially
financed by the NSF Grant~DMS~1301608.}\\ \texttt{Department of Mathematics}\\
\texttt{University of California, Riverside}\\\texttt{\small
mcc@math.ucr.edu}}

\date{}
\maketitle

\centerline{\bf Abstract}

\noindent We establish new estimates on short character sums for arbitrary composite moduli with small prime factors. Our main result improves on the Graham-Ringrose bound for square-free moduli and also on the result due to Gallagher and Iwaniec when the core $q'=\prod_{p|q}p$ of the modulus $q$ satisfies $\log q'\sim \log q$. Some applications to zero free regions of Dirichlet L-functions and the $\rm{P\acute{o}lya}$ and Vinogradov inequalities are indicated.

\bigskip

\noindent  {\bf Introduction.}

\noindent In this paper we will discuss short character sums for moduli with
small prime factors. In particular, we will revisit the arguments of
Graham-Ringrose [GR]  and Postnikov [P]. Our main result is an estimate valid
for general moduli, which improves on the known estimates in certain
situations.

It is well known that non-trivial estimates on short
character sums are important to many number theoretical issues. In particular, they are relevant in establishing density theorems for
the corresponding Dirichlet L-functions.

In the literature, several bounds on short incomplete character sums to some modulus $q$ may be found, depending on the nature of $q$. Burgess' bound applies for moduli $q$ that are cube free, provided the summation interval $I$ has size $N\gg q^{\frac 14 +\epsilon}$. Assuming $q$ has small  prime factors, nontrivial estimates may be obtained under weaker assumptions on $N$. There are two classical results in this aspect, based on quite different arguments. Citing from [IK], the Graham-Rignrose theorem (see [IK], Corollary 12.15) makes the assumptions that $q$ is square-free and
\be\label{2.8.1}N\geq q^{\frac 4{\sqrt{\log\log q}}}+\mathcal P^9\ee
with $\mathcal P$ the largest prime factor of $q$. On the other hand, Iwaniec's generalization of Postnikov's theorem (see [IK], Theorem 12.16) comes with a condition of the form
\be\label{2.8.2}N>(q')^{100}+e^{(\log q)^{3/4}\log\log q}\ee
with $q'=\prod_{p|q}p$ the core of $q$.

The main purpose of this work is to formulate a condition on $N$ as weak as possible, for general modulus $q$ with small prime factors, providing at least subpower savings. That this is possible (assuming $\log N>\phi(\log q)$ for some function $\phi$ satisfying $\frac {\phi(x)}x\to 0$ as $x\to \infty$) was probably known to experts, though no result of this kind seems to appear in the literature.

More specifically, we prove
the following.

\medskip

\noindent{\bf Theorem 5.} {\it Assume $N$ satisfies
$$
q>N>\max_{p|q} p^{10^3}
$$
and
\be\label{2.8.11}
\log N>(\log q)^{1-c} +C \log \Big(2\frac {\log q}{\log q'}\Big) \;\frac {\log q'}{\log\log q}\;,\ee
where $C, c>0$ are some constants, $($e.g. we may take $c= 10^{-3}$ and $C\sim 10^3)$ and $q'=\prod_{_{p|q}}p$.

Let $\chi$ be a primitive multiplicative character modulo $q$ and $I$ an interval of size $N$.
Then
\be\label{2.8.12}\Big|\sum_{x\in I} \chi(x)\Big| \ll N
e^{-\sqrt{\log N}}.\ee
}

\medskip
 \noindent{\it Remark.} By adjustment of the constants $c$ and $C$ in the statement, the bound \eqref{2.8.12} can be improved to $\;N\,e^{ - (\log N)^{1-\epsilon}}$ for any fixed $\epsilon > 0$.

 \medskip
Note that assumption \eqref{2.8.11} of Theorem 5 is implied by the stronger and friendlier
assumption
\be\label{2.8.13}\log N > C\Big(\log \mathcal P + \frac{\log q}{\log\log q}\Big),
\ee
where $\mathcal P=\max_{p|q} p$.
Assumption \eqref{2.8.13} is weaker than Graham-Ringrose's condition
 $$\log N>C\Big(  \log \mathcal P\;+\;\frac {\log q}{\sqrt{\log \log
q}}\Big)\;,$$which moreover assumes $q$ square-free.

\medskip

Many techniques used in the paper are just elaborations of known arguments. Two distinct methods are involved in order to treat small and big prime factors. Small prime powers are dealt with using the standard Postnikov argument combined with Vinogradov's estimate while for big primes we also rely
on Weyl's bound as used in Graham-Ringrose's argument. In addition to the basic techniques introduced in the work of Graham-Ringrose and Postnikov, we introduce one further ingredient which is a (new) mixed character sum estimate (see Theorem \ref{Theorem1}). It allows to merge  more efficiently Postnikov's procedure of replacing multiplicative characters to a powerful modulus by additive characters with a polynomial argument and the Weyl  differencing scheme which is the basis of the Graham-Ringrose analysis. Note that the replacement of $(\log\log q)^{\frac 12}$ in \eqref{2.8.1} by $\log\log q$ is achieved by a more economical variant of the Graham-Ringrose argument (based on a notion of `admissible pair' $(f, q)$ with $f\in\mathbb Z[x]$, $q\in \mathbb Z$). But this is a technical point with no essentially new ideas.

The condition \eqref{2.8.11} in Theorem 5 is the best we could do. But we did not try to optimize the power $1-c$ in the first term nor the saving in \eqref{2.8.12}. The main interest of \eqref{2.8.11} compared with \eqref{2.8.2} is that the assumption
$$\log N>C\log q'$$
is weakened to
$$\log N >C \frac {\log(2\frac {\log q}{\log q'})}{\log\log q}\; \log q'$$
leading to an improvement when $q'$ is relatively large.

This is the place where the mixed character sum (Theorem \ref{Theorem1}) comes into play.

\medskip
Next, we turn to some consequences of Theorem 5 that are elaborated in the last section of the paper.

Following well-known arguments (cf. [\;I\;]), Theorem 5 implies the following zero-free regions for the corresponding Dirichlet L-functions.

\medskip

\noindent{\bf Theorem 10.} {\it Let $\chi$ be a primitive multiplicative character with modulus $q$,
$\mathcal P=\max _{p|q}p,$ $ q'=\prod_{p|q}p,$ and $ K=\frac{\log q}{\log q'}.$ For $T>0$, let $$\theta=c\min\Big(\frac 1{\log
\mathcal P}, \frac {\log\log q'}{(\log q')\log 2K},\frac 1{(\log
qT)^{1-c'}}\Big).$$ Then the Dirichlet L-function
$L(s,\chi)=\sum_n\chi(n) n^{-s}, s=\rho+it$ has no zeros in the region
$\rho>1-\theta, \; |t|<T$, except for possible Siegel zeros.}

\medskip

In the theorem above, one may take $c'=1/10$.
In certain ranges of $q'$, Theorem 10 improves upon Iwaniec's condition [\,I\,]
$$\theta=\min\bigg\{c\frac 1{(\log q T)^{\frac 23}(\log\log q
T)^{\frac 13}}\;, \frac  1 {\log q'}\bigg\}.$$

\medskip
Using the zero-free region above and the result from [HB2] on the effect of a possible Siegel zero, we obtain the following.
\medskip

\noindent
{\bf Corollary 11.}
{\sl Assume $q$ satisfies that $\log p=o(\log q)$ for any $p|q$. If $(a, q)=1$, then there is a prime $P\equiv a(\mod q)$ such that
$P< q^{\frac {12} 5+o(1)}$.}

\medskip

Using Theorem 5, we may also obtain a slight improvement of the following result by Goldmakher ([G], corollary to Theorem 1) on the $\rm{P\acute{o}lya}$-Vinogradov inequality.

\medskip

\noindent(Goldmakher)  {\it Given $\chi(\mod q)$ primitive, with $q$ square-free. Then
\be\label{2.11.1}\Big|\sum_{n<x}\chi(n)\Big|\ll\sqrt q \;\log q\;
\sqrt{\log\log\log q}\left(\frac 1{\log\log q}+\frac {\log\mathcal P}{\log q}\right)^{\frac 14}.\ee}

\medskip

What we obtain is the following.

\medskip

\noindent{\bf Theorem 12.} {\it Let $\chi$ be a primitive multiplicative character with modulus $q$, and let
$\mathcal P$ be the largest prime divisor of $q$, $  q'=\prod_{p|q}p$ and
$K=\frac{\log q}{\log q'}$.  Let $M=(\log q)^{1-c}+\frac {\log q'}{\log\log q'}\log 2K+\log \mathcal P. $ Then
$$\Big|\sum_{n<x}\chi(n)\Big|\ll\sqrt q \sqrt{\log q}\;\sqrt M\sqrt{\log\log\log q}.$$}

In particular, Theorem 12 gives the bound\be\label{2.11.2}\sqrt
q\;\log q\;\sqrt{\log\log\log q}\;\bigg( \frac 1{\sqrt{\log\log q}}+\frac{\sqrt{\log \mathcal P}}{\sqrt{\log q}}\bigg)\;\;
\text{ for arbitrary } q. \ee

\smallskip

Clearly, \eqref{2.11.2} is a stronger bound than \eqref{2.11.1}.

\medskip

The paper is organized as follows. In Section 1, we state the mixed character sum theorem with square-free modulus and indicate where the changes are in the proof for prime modulus. In Section 2, we give a version of Postnikov's Theorem, using it to derive a non-trivial character sum bound for the modulus $q_0^mq_1$,
$q
_1$ square-free. Section 3 contains the notion of admissible pair and an improved version of Graham-Ringrose Theorem. Section 4 is the Graham-Ringrose version of mixed character sum estimate.
Section 5 contains our main theorem, discussion of our assumption and comparison of it with the assumptions in known results. The proof of the main theorem is in Section 6 and Section 7, its applications in Section 8.

Our main interest is the general form of the bounds as a function of the modulus. Constants may often be improved and we did not put emphasis on those.
\bigskip

\noindent{\bf Notations and Conventions.}

1. $e(\theta)=e^{2\pi i\theta}$, $e_p(\theta)=e(\frac {\theta}p).$

2. $\omega(q)=$ the number of prime divisors of $q$.

3. $\tau(q)=$ the number of divisors of $q$.

4. $q'=\prod_{_{p|q}}p$, the core of $q$.

5. $\mathcal P=\mathcal P(q)=\max_{_{p|q}}p$.

6. When there is no ambiguity, $p^{\;\ve} =[p^{\;\ve}]\in \mathbb
Z$.

7. Modulus $q$ is always sufficiently large.

8. $\epsilon, c, C=$ various constants, and $\epsilon$ is particularly small.

9. All characters are non-principal.

10. For polynomials $f(x)$ and $g(x)$ with no common factors, the {\it degree} of $ {f(x)\over g(x)}$ is $\deg f(x)+\deg g(x)$.

11. $A\ll B$ and $A=O(B)$ are each equivalent to that $|A|\leq cB$ for some constant $c$. If the constant $c$ depends on a parameter $\rho$, we use $\ll_\rho$. Otherwise, $c$ is absolute.
\bigskip
\section{ Mixed character sums.}
\begin{theorem}\label{Theorem1}
Let $P(x)\in\mathbb R[x]$ be an arbitrary polynomial of degree
$d\geq 1$, $p$ a sufficiently large prime, $I\subset [1, p]$ an
interval of size \be\label{1.1} |I| > p^{\frac 14 +\kappa} \ee $($for
some $\kappa>0)$ and $\chi$ a  multiplicative character
$(\mod p)$. Then \be\label{1.2} \Big|\sum_{n\in I} \chi(n) e^{i
P(n)}\Big|< c(\kappa)\,d^2\,|I|\; p^{-\;\frac{\kappa^2 }{10(d^2+2d+3)}}. \ee
\end{theorem}

In the proof of Theorem 1, the assumption that $p$ is a prime is only
used in order to apply Weil's bound on complete exponential sums.
(For the Weil's estimate below, see Theorem 11.23 in [IK])

\medskip

 \noindent{\bf Weil's Theorem. } {\it Let $p$ be a prime, $f\in \mathbb Z[x]$ a polynomial of degree $d$, and $\chi$ a  multiplicative
character $(\mod p)$ of order $r>1$. Suppose
$f(\mod p) $ is not an $r$-th power. Then we have
$$\Big|\sum_{x=1}^p\chi(f(x))\Big|\leq d\; \sqrt p.$$}

The assumption of $f(\mod p)$ in Weil's Theorem holds if $f (\mod p) $ has a simple root or a simple pole. For $q$ square-free, one can derive the following theorem.

\medskip

 \noindent{\bf Weil's Theorem'.} {\it  Let $q=p_1\cdots p_k$ be square-free, $f\in \mathbb Z[x]$ a polynomial of degree $d$, and $\chi$ a multiplicative
character $(\mod q)$. Let $q_1 | q$ be such that for any prime $p | q_1$, $f (\mod p) $ has a simple root or a simple pole.
Then
$$\Big|\sum_{x=1}^q\chi(f(x))\Big|\leq d^{\;\omega(q_1)} { q \over \sqrt{q_1} }.$$}

\noindent {\it Proof.} Let  $\chi=\prod_i\chi_i$, where $\chi_i$ is a multiplicative character $(\mod p_i)$. Then $$\Big|\sum_{x=1}^q\chi(f(x))\Big|\leq \prod_{i=1}^k \Big| \sum_{x=1}^{p_i} \chi_i( f(x) ) \Big|\leq \prod_{p_i\mid q_1}d\sqrt{p_i}\prod_{p_i\nmid q_1}p_i=d^{\;\omega(q_1)}\frac q {\sqrt{q_1} }.$$

Therefore, we have the following.

\noindent{\bf  Theorem 1'.} {\it Let $P(x)\in\mathbb R[x]$ be an
arbitrary polynomial of degree $d\geq 1$, $q\in \mathbb Z$ square-free and sufficiently large, $I\subset [1, q]$ an interval of size
\be\label{1.1} |I| > q^{\frac 14 +\kappa} \ee (for some $\kappa>0$)
and $\chi$ a
multiplicative character $(\mod q)$. Then
\be\label{1.2} \Big|\sum_{n\in I} \chi(n) e^{i P(n)}\Big|< c(\kappa)\,d^2\,|I|\;
q^{-c\kappa^2 d^{-2}}\tau(q)^{4(\log d)d^{-2}}. \ee Here $c$ is an absolute constant.}

\noindent{\it Remark 1.1.} In the proof of Theorem 1 in [C1], the assumption that $p$ is a prime is only used to derive display (14) from display (13) by applying Weil's Theorem. For $q$ square-free, the same argument works if Weil's Theorem is replaced by Weil's Theorem'.

\smallskip

\noindent{\it Remark 1.2.} In Theorem 1' if $d<(\log q)^{1/3}$, then the factor $d^2$ can be dropped.
\bigskip

\section{ Postnikov's Theorem.}
An immediate application is obtained by combining Theorem
\ref{Theorem1}' with Postnikov's method (See [P], [Ga], [\,I\,], and
[IK] \S12.6).

\medskip

\noindent{\bf Postnikov's Theorem.} {\it Let $\chi$ be a primitive
multiplicative character $(\mod q)$, $q=q_0^m$. Then
$$\chi(1+q_0u)= e_{q}
\big(F(q_0u)).$$ Here $F(x) \in\mathbb Q[x]$  is a polynomial of the
form  \be\label{2.4} F(x) =BD \Big(x-\frac{x^2}2+\cdots\pm
\frac{x^{m'}}{m'}\Big) \ee with
$$D=\mathop{\prod_{k\leq m'}}_{(k, q_0)=1} k, \quad m'= 2m$$ and
$B\in\mathbb Z, ( B, q_0)=1$.}
(Note that $F(q_0x) \in\mathbb Z[x]$.)

\medskip

\noindent{\it Remark.} In [IK] the above theorem was proved for
$\chi(1+q'u)= e_{q} (F(q'u))$, where $q'=\prod_{p|q}p$ is the core
of $q$. That argument works verbatim for our case.

\begin{theorem}\label{Theorem2}
Let $q=q^m_0 q_1$ with $ (q_0, q_1)=1$ and $q_1$ square-free.

Assume $I\subset [1, q]$ an interval of size \be\label{2.1} |I|>q_0
q_1^{\frac 14+\kappa}. \ee Let $\chi$ be a multiplicative character
$(\mod q)$ of the form
$$
\chi=\chi_0\chi_1
$$
with $\chi_0(\mod q_0^m)$ arbitrary and $\chi_1(\mod q_1)$
primitive. Then \be\label{2.2} \Big|\sum_{n\in I} \chi(n)\Big|\ll
|I| q_1^{-c\kappa^2 m^{-2}} \tau(q_1) ^{c(\log m)m^{-2}}. \ee
\end{theorem}

\begin{proof}
For $a\in [1,q_0], (a, q_0)=1$ fixed, using Postnikov's Theorem, we
write \be\label{2.3} \chi_0(a+q_0x)=\chi_0(a) \chi_0(1+q_0\bar a
x)=\chi_0(a) e_{q_0^m} \big(F(q_0\bar a x)\big), \ee where
$$
a\bar a =1 \quad(\mod q_0^m).
$$
Hence \be\label{2.5} \Big|\sum_{n\in I} \chi(n)\Big|\leq \sum_{(a,
q_0)=1} \Big|\sum_{a+q_0 x\in I} \ e_{q_0^m} \big(F(q_0\bar ax)\big)
\chi_1 (a+q_0x)\Big|. \ee Writing $\chi_1 (a+q_0 x)=\chi_1 (q_0)
\chi_1(a\bar q_0+x), q_0\bar q_0\equiv 1 (\mod q_1)$, the inner sum
in \eqref{2.5} is a sum over an interval $J=J_a$ of size
$\sim\frac{|I|}{q_0}$, and Theorem \ref{Theorem1}' applies.
\end{proof}

\bigskip

\section{ Graham-Ringrose Theorem.}
As a warm up, in this section we will reproduce Graham-Ringrose's
argument. With some careful counting of the {\it bad} set, we are
able to improve their condition on the size of the interval from $q^
{1/\sqrt{\log \log q}}$ to $q^{C/\log \log q}$.

\begin{theorem}\label{TheoremGR}
Let $q\in\mathbb Z$ be square-free, $\chi$ a primitive
multiplicative character $(\mod q)$, and $N<q$. Assume
\begin{enumerate}
\item For all $p | q$, $p<N^{\frac 1{10}}$.

\item $\log N>C\frac {\log q}{\log\log q}.$

\end{enumerate}
Then
\[
    \Big| \sum_{x=1}^N \chi(x) \Big|
    \ll N e^{-(\log N/5)^{3/4}}.
\]
\end{theorem}

\medskip

We will prove the following stronger and more technically stated theorem.

\medskip

\noindent{\bf Theorem 3'} {\it
 Assume $q = q_1 \dots q_r$  with $(q_i,q_{j})=1$ for $i\not = j$, and $q_r$ square-free.
Factor
$$
\chi=\chi_1\ldots \chi_r,$$
where $\chi_i (\mod q_i)$
is arbitrary for $i<r$, and primitive for $i=r$.
 We further assume

\rm{(i)}. {\it For all} $ p | q_r, p > \sqrt{\log q_r}.$

\rm{(ii)}. {\it For all }$i$, $q_i<N ^{\frac 13}$.

\rm{(iii)}. $r < c \log \log q$ for some $c<1/4-\epsilon$.

\smallskip

\noindent {\it Then}
$$
    \Big| \sum_{x=1}^N \chi(x) \Big|
    \ll N e^{-(\log q_r)^{1-c}/\log\log q_r}.
$$}

\smallskip

\noindent{\it Remark 3.1.} Instead of proving Theorem 3, we will prove Theorem 3'. To see that Theorem \ref{TheoremGR}'
implies Theorem \ref{TheoremGR}, we write

$$q=\bar p_1\cdots\bar p_{\ell} \cdot p_1\cdot p_2\cdots,$$
where $$\bar p_1, \dots, \bar p_{\ell}<\sqrt {\log q},\quad\text{and
}\;\; p_1>p_2>\cdots\geq\sqrt {\log q}.$$
Hence \be\label{GR1}\prod_{i=1}^\ell\bar p_i<e^{2\sqrt {\log
q}}<q^{\frac 1{10}}.\ee

Let $q_1=\prod_{i=1}^k p_i$, where we let $k$ be maximum as to ensure that  $q_1<N^{\frac 13}$.

Therefore, $p_{_{k+1}}q_1>N^{\frac 13}$. By (1), $q_1>N^{\frac 13 -\frac
1{10}}>N^{\frac 15}$.

We repeat this process on
$\frac q{q_1}$
to get $q_{2}$ such that $N^{\frac 13}> q_{2}>N^{\frac 15}$.
Then, we repeat it on $\frac q{q_1q_{2}}$ etc. After re-indexing,
we have
$$q_r> q_{r-1}>\dots> q_2 >N^{\frac 15}.$$
Hence $q>\big(N^{\frac 15}\big)^{r-1},$ which together with (2) gives
(iii). Theorem 3' now can be applied. $\quad\square$

\medskip

\noindent{\it Remark 3.2.} It follows from (ii) and the argument in Remark 3.1 that
in the proof of Theorem\ref{TheoremGR}, one may choose $q_r$ satisfying $N^{1/5}<q_r<N^{1/3}$. Hence $\log N\sim \log q_r$.

\medskip

The following definition will be used frequently throughout the rest of the paper.

\noindent{\bf Definition.} Let $p$ be a prime and $f \in \mathbb Z[x]$.  We say $p$ is {\it good} or $f$ is $p$-{\it good}, if $f$ $ \mod p$ has a simple root or a simple pole. Otherwise it is called {\it bad} or $p$-{\it bad}. For $\bar q | q_r$ satisfying $ \bar q> \sqrt {q_r}$, the pair $(f,\bar q)$ is called $q_r$-{\it admissible} (or {\it admissible} when there is no ambiguity) if $$  p > \sqrt{\log q_r}\;\;\text{ for all } \;\;p | \bar q,$$ and $$
\mathop{ \prod_{p|\bar q} }_{\text{$p$ is good }}p\;
>\; { \bar q \over q_r^\tau }\;, \quad\text{ where }
    \tau = { 10 \over \log \log q_r }\;.$$

\noindent{\it Remark 3.3.}
    Let $(f,\bar q)$ be admissible, and let  $\chi$ be primitive mod $\tilde{\tilde q}$, where $\tilde{\tilde q}$ is square-free and a multiple of  $\bar q$. Assume 
    \be\label{GR2}
        \log d < \frac 1 \tau = { \log \log q_r \over 10 }, \text{ where } d=\deg f.
    \ee
    Then we have a bound on the complete sum
    \[
        \left| \sum_{x=1}^{\tilde{\tilde q}} \chi(f(x)) \right|
        <\tilde{\tilde q} \;(\hat q) ^{\;-\frac{3}{10}}< \tilde{\tilde q} (\bar q)^{-\frac 3{10}} q_r^{\frac 3{10}\tau},
    \]where $\hat{q}$ is the product of the good primes $p|\bar q$.

    \begin{proof}[Proof of Remark 3.3.]
     For $p | \bar q | q_r$, our assumptions imply
\be\label{GR3} p^{1/5} > (\log q_r)^{1/10}>d.\ee To prove the remark, we factor $\chi=\chi_1\chi_2$, where $\chi_1$  (respectively, $\chi_2$) is a character $\mod \bar q$ (resp. $\mod  \frac{\tilde{\tilde q}}{\bar q}$).
Weil's estimate gives a bound on the complete sum of $\chi_1$ in the following estimate.
\be\label{GR4}\begin{aligned}
 &\bigg| \sum_{x=1}^{\tilde{\tilde q}} \chi(f(x)) \bigg|\\\leq\;&\bigg| \sum_{x=1}^{\bar q} \chi_1(f(x)) \bigg|\;\bigg| \sum_{x=1}^{\frac{\tilde{\tilde q}}{\bar q}} \chi_2(f(x)) \bigg|
    < { \bar q \over \sqrt{\hat{q}} }\; d^{\;\omega(\hat{q})}\;\frac{\tilde{\tilde q}}{\bar q}={ \tilde{\tilde q} \over \sqrt{\hat{q}} }\; d^{\;\omega(\hat{q})}.\end{aligned}
\ee
where $\hat{q}$ is the product of the good primes $p|\bar q$.  Using
\eqref{GR3}, we bound the character sum above by
\[
    \tilde{\tilde q} \prod_{p | \hat{q}} { d \over \sqrt p }
    < \tilde{\tilde q} \prod_{p | \hat{q}} p^{-3/10}
    = \tilde{\tilde q} \hat{q}^{\;-3/10}.
\]
Since $(f,\bar q)$ is admissible, we have
\[
 \left| \sum_{x=1}^{\tilde{\tilde q}} \chi(f(x))  \right|<  \tilde{\tilde q} \hat{q}^{\;-\frac 3{10}}
    <\tilde{\tilde q} (\bar q)^{-\frac 3{10}} q_r^{\frac 3{10}\tau}. \qedhere
\]
\end{proof}

\begin{proof}[Proof of Theorem 3'.]

We will use Weyl differencing.

Take $M=\big[ \sqrt N\;\big]$. Shifting the interval $[1,N]$ by $y q_1$ for any $1 \leq y \leq M $, we get \[
    \Big| \; \sum_{x=1}^N \chi(x) \;
    -  \sum_{x=1}^N \chi(x + y q_1)  \;\Big|
    \leq 2y q_1 \ll M q_1 .
\]
Averaging over the shifts gives
\be\label{GR5}  \frac 1 N \Big| \sum_{x=1}^N \chi(x) \Big|
    \leq { 1 \over NM } \sum_{x = 1}^N \Big| \sum_{y=1}^M \chi(x + y q_1) \Big|
        + O\left( { M q_1 \over N } \right).\ee
Let
$$\chi_1'=\chi_2\cdots\chi_r.$$

Using the $q_1$-periodicity of $\chi_1$ and Cauchy-Schwarz inequality on the double sum in \eqref{GR5}, we have
\be\label{GR6}
{ 1 \over NM } \sum_{x = 1}^N \Big| \sum_{y=1}^M \chi(x + y q_1) \Big|
    \leq \left[ { 1 \over NM^2 }
        \sum_{y,\;y'=1}^M \;\Big|\sum_{x=1}^N
            \chi_1'\left( {x + q_1 y \over x + q_1 y'} \right)\Big|\;
        \right]^{1/2}.
\ee

For given $(y, y')$, we consider $$f_{y, y'}(x)= {x + q_1 y \over x + q_1 y'}$$ and distinguish among the pairs
$(f_{y, y'}, q_r)$ by whether or not they are $q_r$-admissible. Note that if $(f_{y, y'}, q_r)$ is not admissible, then the product of {\it bad} prime factors of $q_r$ is at least $q_r^\tau$ and this product  must divide $y-y'$. We will estimate the size of the set of bad $(y, y')$ and use trivial bound for the inner sum in \eqref{GR6} corresponding to such bad $(y, y')$.
\be\label{GR7}\begin{aligned}
&\Big|\big\{ (y, y')\in [1, M]^2: (f_{y, y'}, q_r) \text{ is not admissible }\big\}\Big|\\
\leq &\mathop{ \sum_{Q|q_r} }_{Q\geq q_r^\tau}\Big|\big\{(y, y')\in [1, M]^2: Q|\;y-y'\;\big\}\Big|\\
\leq &\mathop{ \sum_{Q|q_r} }_{Q\geq q_r^\tau}{M^2\over Q}<2^{\omega(q_r)}{M^2\over q_r^\tau}
<M^2 q_r^{-{7\over 10}\tau}.
\end{aligned}\ee
(For the second inequality, we note that $M>Q$.)

Hence \eqref{GR6} is bounded by
\be\label{GR8} q_r^{-{7\over 20}\tau}+\Big|{1\over N} \sum_{x=1}^N\chi_1'(f_1(x))\Big|^{1\over 2},\ee
where $f_1$ is the $f_{y,y'}$ with the maximal character sum among all admissible pairs. i.e. \be\label{GR11}\big|\sum_{x=1}^N\chi_1'(f_1(x))\big|=\mathop{\max_{f_{y,y'}}}_{(f_{y, y'}, q_r)\text{  admissible}} \;\;\bigg|\sum_{x=1}^N\chi_1'(f_{y,y'}(x))\bigg|\; .\ee  Thus, there exists $\bar q_1|q_r$, $\bar q_1>q_r^{1-\tau}$ and for any
$p|\bar q_1$, $f_1$ is $p$-good.

To bound the second term in \eqref{GR8}, we will do induction on the number of  characters in the factorization of $\chi$ and first
prove the following.

\noindent
{\it Claim. For $s=1,\cdots, r-1$, denote $\chi_s'=\chi_{s+1}\cdots \chi_r$. Let $f_s(x)$ be of the form $f_s(x)=\prod_j(x-b_j)^{c_j}$, where $b_j, c_j\in \mathbb Z$ and $\deg f_s\leq 2^s$.\\
\noindent Denote  $\bar q_{0}=q_r$. Assume there is $\bar q_{s-1}|q_r$
such that $\bar q_{s-1}\ge q_r^{1-(s-1)\tau}$ and $(f_s, \bar q_{s-1})$ is admissible.  Then
\be\label{GR9} {1\over N} \;\Big|\sum_{x=1}^N\chi_s'(f_s(x))\Big|\leq q_r^{-{\;\tau\over 5}\cdot\frac 12}+\Big|{1\over N}\sum_{x=1}^N\chi_{s+1}'(f_{s+1}(x))\Big|^{\frac 12},\ee
where $f_{s+1}$ is of the same form as $f_s$ with $\deg f_{s+1}\leq 2^{s+1}$, and there is $\bar q_{s}|q_r$
such that $\bar q_{s} >q_r^{1-s\tau}$ and $(f_{s+1}, \bar q_{s})$ is admissible.
}


\noindent
{\it Proof of Claim.}

As before, the $q_{s+1}$-periodicity of $\chi_{s+1}$ and Cauchy-Schwarz inequality give a bound on the character sum in the left-hand-side of \eqref{GR9} by
\be\label{GR10}
  \left[ { 1 \over NM^2 }
        \sum_{y,\;y'=1}^M \;\Big|\sum_{x=1}^N
            \chi_{s+1}'\left( {f_s(x + q_{s+1} y) \over f_s(x + q_{s+1} y')} \right)\Big|\;
        \right]^{1/2}.
\ee
Set
$$f_{s+1}(x)={f_s(x + q_{s+1} y) \over f_s(x + q_{s+1} y')}\;,$$ where $(y, y')$ is chosen among all good pairs as in \eqref{GR11}, such that the inner character
sum in \eqref{GR10} is the maximum.

We want to bound the set of bad $(y, y')$. For $p|\bar q_{s}$,
$$f_s(x)=(x-a)^{\epsilon}\prod_j(x-b_j)^{c_j}\;\text{ for some $\epsilon\in \{-1, 1\}$, where } a\not= b_j \;\mod p.$$
Hence$$f_{s+1}(x)=\bigg({x+q_{s+1}y-a\over x+q_{s+1}y'-a}\bigg)^{\epsilon}\prod_j\bigg({x+q_{s+1}y-b_j\over x+q_{s+1}y'-b_j}\bigg)^{c_j}\;.$$
For $y\not= y' \;\mod p$, if $a-q_{s+1}y$ is not a simple root or pole, then
$$a-q_{s+1}y= b_j-q_{s+1}y'\quad \mod p$$ for some $j$. Therefore, by the same reasoning as for \eqref{GR7},
\be\label{GR12}\begin{aligned}&\;\Big|\big\{ (y, y')\in [1, M]^2: (f_{s+1}, \bar q_s) \text{ is not admissible }\big\}\Big|\\
\leq&\mathop{ \sum_{Q|\bar q_s} }_{Q> q_r^\tau}\Big|\big\{(y, y')\in [1, M]^2:\forall p|Q, f_{s+1} \text{ is }p\text{-bad }\big\}\Big|\\
\leq & \mathop{ \sum_{Q|\bar q_s} }_{Q> q_r^\tau}{M^2\over Q}\; \big(2^s\big)^{\omega(Q)}=M^2\mathop{ \sum_{Q|\bar q_s} }_{Q> q_r^\tau}
{\big(2^s\big)^{\omega(Q)}\over Q}\;.
\end{aligned}
\ee(In the above bound, the factor $2^s$ comes from the choices of $b_j$.)

By assumptions (i) and (iii),
\be\label{GR15}{\big(2^s\big)^{\omega(Q)}\over Q}\leq \prod _{p|Q} {2^r\over p}<\prod_{p|Q}{1\over \sqrt p}={1\over\sqrt Q}<q_r^{-{\tau\over 2}},\ee
and \eqref{GR12} is bounded by
$$M^2\;{2^{\omega(q_r)}\over q_r^{\tau/2}}<M^2 q_r^{-\tau/5},$$
and the claim is proved.$\;\;\;\square$

 At the last step of our induction, we are bounding
 \be\label{GR13}\Big|\sum_{x=1}^N\chi_r(f_{r-1}(x))\Big|,\ee
where  $f_{r-1}$ is of the form as in the claim with $\deg f_{r-1}\leq 2^{r-1}$ and there is $\bar q_{r-2}|q_r$ such that $\bar q_{r-2}>q_r^{1-(r-2)\tau}>\sqrt q_r$ and
 $\forall p|\bar q_ {r-2}$ is good. In particular, $(f_{r-1}, \bar q_{r-2})$ is admissible and Remark 3.3 applies. \big(Note that $ \bar q_{r-2}^{\;\;\;\;\;\;-\frac 3{10}} \;q_r^{\;\;\frac 3{10}\tau}<\bar q_{r-2}^{\;\;\;\;\;\;-(\frac 3{10}-\frac 35\tau)}<\bar q_{r-2}^{\;\;\;\;\;\;-\frac 14}$.\big) Hence, we have
 $$\Big|\sum_{x=1}^N\chi_r(f_{r-1}(x))\Big|<N\bar q_{r-2}^{\;\;\;\;\;\;-{1\over 4}}<Nq_r^{\;\;-{1\over 8}},$$
 and we reach the final bound
 \be\label{1.28.2013}\begin{aligned}&\frac 1 N \Big| \sum_{x=1}^N \chi(x) \Big|\\\leq \;&q_r^{-7\tau/ 20}+q_r^{-\tau/ 10}+\cdots+\Big( q_r^{-\tau/ 5}\Big)^{1/ 2^{r-1}}+\Big( q_r^{-1/ 8}\Big)^{1/ 2^{r-1}}+O\left(\frac{Mq_1}{N}\right)\\
 \ll \;&\Big( q_r^{-\tau/ 5}\Big)^{1/ 2^{r-1}}+\Big( q_r^{-1/ 8}\Big)^{1/ 2^{r-1}}\\
 \ll \;&\; q_r^{-{2\over \log\log q_r}\cdot  {1\over 2^{c\log\log q_r}}}\\
 =\;\;&e^{-{\log q_r\over \log\log q_r\;(\log q_r)^c}}\\
 <\;\;&e^{-(\log q_r)^{1-c}/\log\log q_r}.\end{aligned}\ee
\end{proof}
The proof of Theorem 3' also gives an argument for the following theorem.

\smallskip

\noindent{\it Remark 3.4.} In Theorem 3 the saving of $e^{(\log N/5)^{3/4}}$ is certainly better than the
$e^{\sqrt{\log N}}$ we use in the theorems thereafter (optimizing the exponent of $\log N$ is not our focus).

\medskip

\noindent{\bf Theorem 3''} {\it
 Assume $q = q_1 \dots q_r$  with $(q_i,q_{j})=1$ for $i\not = j$, and $q_r$ square-free.
Factor
$
\chi=\chi_1\ldots \chi_r,$
where $\chi_i (\mod q_i)$
is arbitrary for $i<r$, and primitive for $i=r$.\\
 We further assume

\rm{(i)}. {\it For all $ p | q_r, p > \sqrt{\log q_r}.$

\rm{(ii)}. For all $i$, $q_i<N^{1/3} $.

\rm{(iii)}. $r < c \log \log q$ for some $c<1/4-\epsilon$.\\
{\it Let $$f(x)=\prod_j(x-b_j)^{c_j}, \;\;\; c_1\in\{-1, 1\}, \;\; d=\deg
f=\sum |c_j|.$$ Suppose that
 $(f, q_r)$ is {\rm admissible} $($as defined after the statement of Theorem 3'$)$. Furthermore,
 assume

 \rm{(iv)}. $d=\deg f<\big(\log q_r\big)^{\frac 18}$.

\noindent {\it Then}
$$
    \Big| \sum_{x=1}^N \chi(f(x)) \Big|
    \ll N e^{-(\log q_r)^{1-c}/\log\log q_r}$$}

\smallskip

\noindent{\it Remark 3.5.} To prove Theorem 3'', one only needs to modify the proof of Theorem \ref{TheoremGR}' slightly by multiplying ${ M^2\over Q}$ by $d^{\omega(Q)}$
 in \eqref{GR7} and replacing $2^s$ (respectively, $2^r$) by
 $2^{s-1}d$ (resp. $2^{r-1}d$) in \eqref{GR12} (resp. \eqref{GR15}).

 \smallskip

\section{ Graham-Ringrose for mixed character sums.}
The technique used to prove Theorem \ref{Theorem1}' may be combined
with the method of Graham-Ringrose for Theorem \ref{TheoremGR}' to bound short mixed character sums
with highly composite modulus (see also \cite{IK} p. 330--334).

Let $q=q_1 \ldots q_r$ with $(q_i,q_{j})=1$ for $i\not = j$, and $q_r$ square-free, such that (i) and (iii) of Theorem \ref{TheoremGR}' hold .

Let
$$
\chi=\chi_1\ldots \chi_r,$$
where $\chi_i (\mod q_i)$
is arbitrary for $i<r$, and primitive for $i=r$.

Let $I\subset [1, q]$ be an interval of size $N<q$, and let $f(x) =\alpha_d x^d +\cdots+\alpha_0\in \mathbb R[x]$ be an arbitrary polynomial
of degree $d$.

Assuming (ii) of Theorem \ref{TheoremGR}' and an appropriate assumption on $d$, we establish a bound on \be\label{3.1} \sum_{x\in I} \chi(x)
e^{if(x)}. \ee

The case $f=0$ corresponds to Theorem \ref{TheoremGR}'. The main idea to bound \eqref{3.1} is as follows. First, we repeat part of the proof of Theorem
\ref{TheoremGR}' in order to remove the factor $e^{if(x)}$ at the cost of obtaining a character sum with polynomial argument. Next, we invoke Theorem \ref{TheoremGR}'' to estimate these sums.

Write $q=q_1Q_1$ with $Q_1=q_2 \ldots q_r,$ and denote $ \mathcal Y_1=
\chi_2\ldots \chi_r$.

Choose $M\in\mathbb Z$ such that \be\label{3.27} M\cdot \max q_i<N,\quad\text{and }\quad M<\sqrt N.\ee
Using shifted product method as in (3.5), we have
\begin{alignat}{2}\label{3.2}
\sum_{x\in I} \chi(x) e^{if(x)} =\frac 1M\sum_{\substack {x\in I\\ 0\leq y< M}}& \chi(x+q_1 y) e^{if(x+q_1y)}+O(q_1 M),\\
 \frac 1M\Big|\sum_{\substack {x\in I\\ 0\leq y< M}} \chi(x+q_1 y) e^{if(x+q_1y)}\Big|&= \frac 1M \Big|\sum_{\substack{x\in I\\ 0\leq y<M}}
 \chi_1(x) \mathcal Y_1 (x+q_1 y) e^{if(x+q_1 y)}\Big|\notag\\
&\label{3.3} \leq \frac 1M\sum_{x\in I} \Big|\sum_{0\leq y<M}
\mathcal Y_1 (x+q_1 y) e^{if (x+q_1 y)}\Big|.
\end{alignat}
Next, we write
$$
f(x+q_1 y)=f_0(x)+f_1(x)y+\cdots+ f_d(x)y^d.
$$
Subdivide the unit cube $\mathbb T^{d+1}$ into cells
$U_\alpha= B\big(\xi_\alpha, \frac 1{M^{d+1}}\big)\subset \mathbb
T^{d+1}$, $\xi_\alpha \in \mathbb T^{d+1}$.

Denote
$$
\Omega_\alpha =\{x\in I: \big(f_0(x), \ldots, f_d(x)\big)\in
U_\alpha \;\;\mod 1\}.
$$
Hence, for $x\in \Omega_\alpha$
\begin{alignat}{2}\label{3.4}
f(x+q_1y)&= \xi_{\alpha, 0} +\xi_{\alpha, 1} y+\cdots+ \xi_{\alpha, d} y^d+O\Big(\frac 1M\Big)\notag\\[6pt]
e^{if(x+q_1 y)}&=e^{i(\xi_{\alpha, 0}+\cdots+ \xi_{\alpha, d} y^d)}+O\Big(\frac 1M\Big).
\end{alignat}

The number of cells is

\be\label{3.5}
\sim\big(M^{d+1}\big)^{d+1}.
\ee
Substituting \eqref{3.4} in \eqref{3.3} gives

\be\label{3.6} \begin{aligned}&\frac 1M\sum_{x\in I} \Big|\sum_{0\leq y<M}
\mathcal Y_1 (x+q_1 y) e^{if (x+q_1 y)}\Big|\\ =&\frac 1M
\sum_\alpha\sum_{x\in\Omega_\alpha} \Big|\sum_{0\leq y\leq M}
C_\alpha (y) \mathcal Y_1 (x+q_1y)\Big|+ O\Big(\frac NM\Big),\end{aligned} \ee where
$|C_\alpha(y)| =1$.

Next, applying H\"older's inequality to the triple sum in \eqref{3.6} with $k\in\mathbb N$, we have
$$\begin{aligned}&
\sum_\alpha \sum_{x\in\Omega_\alpha} \Big|\sum_{0\leq y\leq M}
C_\alpha (y) \mathcal Y_1 (x+q_1y)\Big| \\\leq &N^{1-\frac 1{2k}}
\Big(\sum_\alpha \sum_{x\in I} \Big|\sum_{0\leq y\leq M} C_\alpha
(y) \mathcal Y_1 (x+q_1y)\Big|^{2k}\Big)^{\frac 1{2k}}.\end{aligned}
$$
Therefore, up to an error of $O(\frac NM)$, \eqref{3.6} is bounded by
$$\text{}\qquad N\; \bigg[\frac 1{NM^{2k}}\;\big(
M^{d+1}\big)^{d+1} \sum_{0\leq y_1, \ldots y_{2k} <
M}\Big|\sum_{x\in I} \mathcal Y_1 \Big(\frac {(x+q_1 y_1)\cdots
(x+q_1 y_k)} {(x+q_1 y_{k+1}) \cdots (x+q_1 y_{2k})}\Big)\Big|\bigg]
^{\frac 1{2k}}$$
\be\label{3.8} =N\; \big(M^{d+1}\big)^{\frac
{d+1}{2k}}\; \bigg[ \frac 1{NM^{2k}}
 \sum_{0\leq y_1, \ldots, y_{2k}<M} \Big|\sum_{x\in I} \mathcal Y_1
\big(R_{y_1, \ldots , y_{2k}}(x) \big)\Big|\bigg]^{\frac
1{2k}},\quad
\ee

where
$$
R_{y_1, \ldots, y_{2k}} (x) =\frac {(x+q_1y_1)\cdots (x+q_1 y_k)}{(x+ q_1 y_{k+1})\cdots (x+q_1 y_{2k})}.
$$

To bound the double sum in \eqref{3.8}, we apply Theorem \ref{TheoremGR}'' with
$f(x)=R_{y_1, \ldots, y_{2k}} (x)$ for those tuples $(y_1,\ldots, y_{2k})\in [0,M-1]^{2k}$ for
which $(R_{y_1, \ldots, y_{2k}}, q_r)$ is admissible. For the other tuples, we use the trivial bound.
If $(R_{y_1, \ldots, y_{2k}}, q_r)$ is not admissible, then there is a divisor $Q|q_r$, $Q> q_r^{\tau}$, such that
for each $p|Q$, the set $\{\pi_p(y_1), \cdots, \pi_p(y_{2k})\}$ has at most $k$ elements. Here $\pi_p$ is the
natural projection from $\mathbb Z$ to $\mathbb Z/p\mathbb Z$. We will estimate the contributions of $(y_1, . . . , y_{2k})$ for which $(R_{y_1, \ldots, y_{2k}}, q_r)$ is not admissible by distinguishing the tuples $(y_1,\ldots, y_{2k})$ according to the relative size of $Q$ and its prime factors.

\smallskip

\noindent (a). {\it Suppose that there is $p|Q$ with $p>\sqrt M$.}

Then the number of $p$-bad tuples $(y_1,\ldots, y_{2k})$ is bounded by
$$\begin{pmatrix}{2k}\\{k} \end{pmatrix} M^k k^k\Big(1+{M\over p}\Big)^k<(4k)^k M^{\frac 32 k}.$$Indeed, one chooses a set $I$ of $k$ indices and specify the corresponding $y_j$; for the other indices, $\pi_p(y_j)$ is taken within $\{\pi_p(y_j):j\in I\}$,
and summing over the prime divisors of $q_r$ gives
\be\label{3.55}\omega(q_r) (4k)^k M^{\frac 32 k}<M^{\frac 74k},\ee
provided
\be\label{3.50} 50\leq k<M^{\frac 15},\ee and
\be\label{3.51} \log q_r<M.\ee

\smallskip

\noindent (b). {\it Suppose $Q>M$ and $p\leq \sqrt M$ for each $p|Q$. }

Take $Q_1|Q\;$ such that $\sqrt M<Q_1\leq M$. The number of tuples $(y_1,\ldots, y_{2k})$ that are $p$-bad
for each $p|Q_1$ is at most
$$\begin{aligned}&\Big (\frac M{Q_1}\Big)^{2k}\prod_{p|Q_1} \begin{pmatrix}{2k}\\{k} \end{pmatrix}\; p^k\;k^k\\<\;&\Big (\frac M{Q_1}\Big)^{2k}\prod_{p|Q_1}
(4kp)^k< \;(ck)^{k\omega(Q_1)}{M^{2k}\over Q_1^{ k}}<\;{M^{2k}\over Q_1^{\frac k3}}<\;M^{\frac {11}6k},\end{aligned}$$
provided
\be\label{3.52} k<\min_{p|q_r} p^{\frac 13}.\ee

Summing over all $Q_1$ as above gives the contribution
\be\label{3.53} M^{\frac {11}6k+1}<M^{\frac {15}8 k}.\ee

\noindent (c). {\it Suppose $q_r^{\tau}<Q<M$.}

The number of tuples $(y_1,\ldots, y_{2k})$  that are $p$-bad for
all $p|Q$ is at most $M^{2k}/ Q^{\frac k3}$, and summation over these
$Q$ gives the contribution \be\label{3.54} 2^{\omega(q_r)}\;
{M^{2k}\over Q^{\frac k3}}<{M^{2k}\over q_r\;^{\frac 14 k\tau}}.\ee

Hence, in summary, the number of $(y_1,\ldots, y_{2k})$ for which
$(R_{y_1, \ldots, y_{2k}}, q_r)$ is not admissible is at most
$$M^{2k}(M^{-\frac k8}+
q_r^{-\frac 14k\tau}).$$
From \eqref{3.2}-\eqref{3.3} and \eqref{3.6}-\eqref{3.8} we obtain the estimate
$$ \sum_{x\in I} \chi(x)
e^{if(x)}<N\big(M^{d+1}\big)^{\frac{d+1}{2k}}\Big[M^{-\frac k8}+q_r^{-\frac 14 k\tau}+e^{-\sqrt{\log q_r}}
\;\Big]^{\frac 1{2k}}$$
using Theorem \ref{TheoremGR}'' for the contribution of {\it good} tuples $(y_1,\ldots, y_{2k})$. Here
we need to assume
\be\label{3.55}  2k<\big(\log q_r\big)^{\frac 18},\ee
which also implies \eqref{3.50} and \eqref{3.52}, under assumption (i) and if \eqref{3.51} holds.

Take  $k=50 d^2$, assuming \be\label{3.57}d<\frac 1{10}\Big(\log q_r\Big)^{\frac 1{16}},\ee
(which implies \eqref{3.55},) then
$$ \begin{aligned}\sum_{x\in I} \chi(x)e^{if(x)}< &\;
N M^{\frac{(d+1)^2}{2k}}\Big ( M^{-\frac 1{16}}+e^{-\frac{\sqrt{\log q_r}}{2k}}\Big)\\
< &\; N\bigg(M^{-\frac 1{1600}} +\bigg( \frac {M^{(d+1)^2}}{e^{\sqrt{\log q_r}}}\bigg)^{\frac 1{100d^2}}\bigg).\end{aligned}$$

Choose $$M=\Big[\exp\Big(\frac{\sqrt{\log {q_r}}}{2(d+1)^2}\Big)\;\Big].$$ (So \eqref{3.51} is also satisfied.) We have
$$ \sum_{x\in I} \chi(x)
e^{if(x)}< N e^{-\frac {\sqrt{\log q_r}}{200d^2}}.$$

Thus we proved

\begin{theorem}\label{Theorem3}
Assume $q = q_1 \dots q_r$  with $(q_i,q_{j})=1$ for $i\not = j$, and $q_r$ square-free.
Factor
$
\chi=\chi_1\ldots \chi_r,$
where $\chi_i (\mod q_i)$
is arbitrary for $i<r$, and primitive for $i=r$.

 We further assume

{\rm(i)}. For all $ p | q_r, p > \sqrt{\log q_r}.$

{\rm(ii)}. For all $i$, $q_i< N^{1/3} $.

{\rm(iii)}. $r < c \log \log q$ for some $c<1/4-\epsilon$.

Let $f(x) \in\mathbb R[x]$ be an arbitrary polynomial of degree $d$. Assume
$$d<\frac 1{10} \Big(\log q_r\Big)^{\frac 1{16}}.$$
Then \be\label{3.17} \Big|\sum_{n\in I} e^{if(n)} \chi(n)\Big| <
CN e^{-\frac {\sqrt{\log q_r}}{200d^2}},
\ee where $I$ is an interval
of size $N$.
\end{theorem}

Combined with Postnikov (as in the proof of Theorem \ref{Theorem2}), Theorem \ref{Theorem3} then implies

\medskip

\noindent{\bf Theorem 4'}
{\it Suppose $q = q_0 \dots q_r$  with $(q_i,q_{j})=1$ for $i\not = j$, and $q_r$ square-free. Assume $\bar q_0|q_0$ and $q_0|(\bar q_0)^m$ for some $m\in\mathbb N$, and
$$m<\frac 1{20} \Big(\log q_r\Big)^{\frac 1{16}}.$$
Factor
$
\chi=\chi_0\ldots \chi_r,$
where $\chi_i (\mod q_i)$
is arbitrary for $i<r$, and primitive for $i=r$.

 We further assume

{\rm(i)}. For all $ p | q_r, p > \sqrt{\log q_r}.$

{\rm(ii)}. For all $i$, $q_i< \big(N/\bar q_0\big)^{1/3} $.

{\rm(iii)}. $r < c \log \log q$ for some $c<1/4-\epsilon$.

Then
\be\label{3.18} \Big|\sum_{n\in I}\chi(n)\Big|< CN e^{-\frac {\sqrt{\log q_r}}{800m^2}}, \ee
 where $I$ is an interval
of size $N$.}

Note that for Theorem 4' to provide a nontrivial estimate, we should assume at least
$$
r\ll \log\log q_r
$$
and
$$
\log m\ll \log\log q_r.
$$

\section{The main theorem.}

Theorem \ref{Theorem3}' as a consequence of Theorem \ref{Theorem3}
was stated mainly for expository reason. (cf. \S7. Proposition 7.) Our
goal is
 to develop this approach further in order to prove the following stronger result.

\begin{theorem}\label{Theorem8}
Assume $N$ satisfies
\be\label{15.1}
q>N>\max_{p|q} p^{10^3}
\ee
and
\be\label{15.2}
\log N>(\log q)^{1-c} +C \log \Big(2\frac {\log q}{\log q'}\Big) \;\frac {\log q'}{\log\log q}\;,\ee
where $C, c>0$ are some constants, $($e.g. we may take $c= 10^{-3}$ and $C\sim 10^3)$ and $q'=\prod_{_{p|q}}p$.

Let $\chi$ be primitive $(\mod q)$ and $I$ an interval of size $N$.
Then
\be\label{15.3} \Big|\sum_{x\in I} \chi(x)\Big| \ll N
e^{-\sqrt{\log N}}. \ee

\end{theorem}

\medskip

We will prove Theorem 5 in the next two sections. In this section, we make some further technical specifications which will be important in the proof. Also, we  discuss assumption \eqref{15.2} and compare it with the assumptions in known results. Some of the remarks at the end of this section (Remarks 5.2 and 5.3) will be used in the proof as well.

\smallskip

\noindent{\it Claim.} We may make the following assumptions.
\be\label{15.4}(1.)\;\; q'>N^{\frac 1{200}}\qquad\qquad\qquad\qquad\qquad\qquad\qquad\qquad\qquad\qquad
\qquad\qquad\qquad\qquad\ee
$$(2.)\;\; q=Q\; q_r=Q_1\cdots Q_{r-1}\;q_r, \text{ where }  (Q_i, Q_j)=(Q_i, q_r)=1,
 \qquad\qquad\qquad\qquad\qquad\qquad\qquad\qquad\qquad\qquad\qquad\qquad\qquad\qquad $$
 \be\label{15.9}  r=1+10 \bigg[\frac{\log Q \overline{}'}{\log N}\bigg],\ee
 (where $Q'$ is the core of $Q$,) and
 \be\label{15.8}q_r=q_0^m, \;\; q_0 \text{ square-free }\ee with
\be\label{15.5} e^{ (\log N)^{\frac 34}}<q_0<N^{\frac 1{10}},\ee
\be\label{15.6} m\leq (\log N)^{3c}.\ee

Also, the core of
$Q_s$ satisfies \be\label{15.7} Q_s'<N^{\frac 15}\;\text{ for }
s=1,\ldots, r-1.\ee

Moreover, we may assume
\be\label{1.10.2013} \forall p|q_r,\;\; p>\sqrt{\log q_r}.\ee

\noindent (3.) There exist $q_1, \cdots, q_{r-1}$, $(q_i, q_j)=(q_i, q_r)=
1$,
\be\label{15.11}\max_i q_i<N^{1/2},\ee such that
\be\label{15.10}q=Q_1\cdots Q_{r-1}q_r\; |\; q_1^{m_1}\cdots q_{r-1}^{m_{r-1}}q_r,\ee
with\be\label{15.12} m_s=10\bigg[\frac{\log Q_s}{\log N}\bigg].\ee

\noindent{\it Proof of Claim.}

To verify Assumption (1), we will obtain the bound \eqref{15.3} for the case $q'\leq N^{\frac 1{200}}$ by using
 Theorem 12.16 in \cite{IK}. The latter provides the following bound
\be\label{15.13}
\Big|\sum_{M<x\leq \widetilde{M}} \chi(x)\Big|< C^{s(\log s)^2} M^{1-\frac{c}{s^{2}\log s}}
\ee
with $s=\frac{\log q}{\log M}$, assuming that $q'^{100}<M<\widetilde{M}\leq 2M$.  This gives a nontrivial bound $M^{1-\frac{c}{s^{2}\log s}}$ provided $\log M \gtrsim(\log q)^{\frac{3}{4}+\epsilon}$. We will use this result by dividing our interval $[1,N]$ dyadically.

Let $M_1=N e^{-\sqrt{\log N}}$, $M_i=2^{i-1}M_1$ and $s_i=\frac{\log q}{\log M_i}$, for $i=2, \cdots, m=\sqrt{\log N}$.

We divide $[1,N]$ into subintervals $$[1, N]=[1, M_1]\;\bigcup_{i=1}^{m-1} \;(M_{i}, M_{i+1}],$$ and note that  $M_i>N^{1/2}>q'^{100}$, and $\log M_i\sim\log N$. Hence $s_i\sim\frac{\log q}{\log N}:=r$ for $i\geq 1$.

We bound the character sum $\sum_{i=1}^N \chi(x)$ by bounding subsum over each subinterval, using the trivial bound for the interval $[1, M_1]$ and  Theorem 12.16 in \cite{IK} for the intervals $(M_{i}, M_{i+1}]$. It is straightforward to check that the sum of all bounds is bounded by$$N e^{-\sqrt{\log N}}+m\; M_m^{1-\frac{c}{r^{2}\log r}} \ll N e^{-\sqrt{\log N}},$$ if
$\log N \gtrsim(\log q)^{\frac{4}{5}+\epsilon}$, which follows from \eqref{15.2}.

To see Assumption (2), we first note that
\be\label{15.14} \prod_{\nu_p >(\log N)^{3c}}p<q^{(\log N)^{-3c}}<e^{(\log N)^{1-c}}<q'^{\frac 1{200}},\ee
where $\nu_p$ is the exponent of $p$ in the prime factorization of $q$. In the display above, the second inequality follows from \eqref{15.2} (which implies that $\log q<(\log N)^{1+2c}$), and the last inequality follows from Assumption (1) in the claim and provided $\log q'$ is sufficiently large.

From \be\label{1.10.2013.3}\prod_{m=1}^{(\log N)^{3c}}\big(\prod_{\nu_p=m}p\big)>q'^{\frac{99}{100}},\ee there exists $m\leq (\log N)^{3c}$ such that
$$\prod_{\nu_p=m}p>(q')^{\frac 12(\log N)^{-3c}}>e^{\frac 1{400}(\log N)^{1-3c}}>e^{(\log N)^{\frac 34}}.$$
The second inequality in Assumption \eqref{15.1}  that $\;\max_{p|q} p<N^{10^3}$ ensures that we may take $$q_0\Big|\prod_{\nu_p=m}p\;\;\text{ such that }\;\;q_0<N^{\frac 1{10}}.$$
Therefore, there exists a $q_r$ which satisfies \eqref{15.8}-\eqref{15.6}.

To see \eqref{1.10.2013}, we note that in the inner product in \eqref{1.10.2013.3}, we may impose the condition that $p>\sqrt{\log q_r}$, because
\be\label{1.10.2013.2}\prod_{p<\sqrt {\log q_r}} p<e^{2\sqrt {\log q_r}}<e^{(\log N)^{\frac 23}}<q'^{\frac 1{200}}.\ee
(The second inequality follows from \eqref{15.8}-\eqref{15.6}.)

Write $q=Q\; q_r$, and $$Q=\prod p_i^{\nu_i},\quad\text{ where } \; \nu_i:=\nu_{p_i}\;\text{ and } \nu_1\geq \nu_{2}\geq\cdots.$$ Let $Q'=\prod_{\nu_i\geq 1}p_i$ be the core of $Q$, and factor
$$Q'=Q_1'\cdots Q_{r-1}' \text{ such that } Q_s'<N^{\frac 15} \text{ and } r=1+10 \bigg[\frac{\log Q \overline{}'}{\log N}\bigg].$$
For each $s$, define $Q_s =\prod_{p|Q_s'}p^{\;\nu_p}$ and $q_s=\prod_{p|Q_s'} p^{\;\bar\nu_p}$ as follows
\be\label{15.15}
\bar\nu_p =\begin{cases} \big[\frac{\nu_p}{m_s}\big]+ 1, \text {  if } \nu_p> m_s\\ 1,\quad \text { otherwise}.\end{cases}
\ee
Denote
$$
m_s = 10\frac {\log Q_s}{\log N}.
$$
It follows that $Q_s|(q_s)^{m_s}$ and $q_s < Q_s' Q_s^{\frac 2{m_s}}< N^{\frac 12}$, which are \eqref{15.11}-\eqref{15.10}.

\medskip

\noindent{\it Remark 5.1.} Assumption  \eqref{15.2}  can be reformulated as
\be\label{15.16}
\Big(3\;\frac {\log q}{\log q' N}\Big)^{10\frac{\log N q'}{\log N}} <(\log N)^c.
\ee

\noindent{\it Remark 5.2.} Using \eqref{15.4}, \eqref{15.16} and the
inequality of arithmetic and geometric means, one can show that
\be\label{15.17} \prod_{i=1}^{r-1} m_i < \big(\log q_0\big)^{\frac
1{75}}.\ee

\noindent{\it Remark 5.3.} It is easy to check that \eqref{15.2} and \eqref{15.5} imply
 \be\label{15.18}r < 10^{-3} \log \log q_0.\ee

\noindent{\it Remark 5.4.}
If $\log q'\leq \log N$, \eqref{15.16} becomes
$$
\log N>(\log q)^{1-c},
$$
which is similar to Theorem 12.16 in \cite{IK}.

\noindent{\it Remark 5.5.} If $q=q'$ (i.e. $q$ is square-free),
condition \eqref{15.16} becomes
$$
\frac{\log q}{\log\log q} < c\log N.
$$
This is slightly better than Corollary 12.15 in \cite{IK} and essentially optimal in view of the Graham-Ringrose argument.

\section{The proof of Theorem \ref{Theorem8}.}

The following lemma is the technical part of the inductive step. It is based on the techniques proving Theorem 2 and Theorem 4, which are shifting product and averaging (Graham-Ringrose), replacing $\chi_i$ by an additive character of a polynomial(Postnikov), and the mixed character technique to drop the additive character.

\medskip

\noindent{\bf Lemma 6.} {\it Assume

\noindent $(a).\;$ $q=q^{m}_{1} \hat q, $ where $\hat q=\hat {\hat q}q_r$, with $q_1, \hat {\hat q}, q_r$ mutually coprime.

\noindent $(b).\;$ $\chi=\chi_1 \hat \chi,$ with $
\chi_1 (\mod q^{m}_1)$ and $ \hat \chi (\mod \hat q)$.
$$(c).\;\; f(x)
=\prod^\beta_{\alpha=1} (x-a_\alpha)^{d_\alpha},\;  d=\sum|d_\alpha|, \text{ with } a_\alpha
\in\mathbb Z \text{ distinct and }d_\alpha\in\mathbb Z\setminus\{0\}
.\qquad\qquad\qquad\qquad\qquad\qquad\qquad
$$
\noindent $(d).\;$ there exists $ \bar q|q_r$ such that for each $p|\bar q$, $p>\sqrt {\log q_r}$ and $f$ is $p$-good.

\noindent $(e).\;$ $I$ an interval of length $N$, $\;q_1^2<N<q$,  $\;1440\cdot m^2d<(\log N)^{\frac 1{5}}$.

\noindent $(f).\;$ $M\in \mathbb Z$, $\;\log N+\log \bar q< M< N^{\frac 1{10}}, \;M<\bar q^{\;\tau}$.

 Then
\be\label{6.4}
 \frac{1}{N} \Big|\sum_{x \in I} \chi \big(f(x)\big)\Big|<
M^{-1/15} +  M^{^{\frac 1{60}}} \Big|\frac 1N\sum_{x \in I} \hat \chi
\big(f_1(x)\big)\Big|^{\frac 1{60m^2}}, \ee where $f_1(x)$ is of the
form \be\label{6.5} f_{1}(x) = \frac{\prod^{k}_{\nu=1} f(x+q_{1}
t_{\nu})} {\prod^{2k}_{\nu=k +1} f(x+q_{1} t_{\nu})} =
\prod^{\beta'}_{\alpha'=1} (x-b_{\alpha'})^{d_{\alpha'}}\ee with
$b_{\alpha'}, d_{\alpha'} \in \mathbb{Z}$, $2k=60 m^2$ and
\be\label{6.6} d^{(1)}:=\sum\mid d_{\alpha'}\mid \leq 60\;d\;m^2.\ee
Furthermore, $(f_1, \bar q)$  is admissible.}

\noindent{\it Proof.} Take $t\in[1,M]$. Clearly, \be\label{4.6.4}
\chi\big(f(x+tq_1)\big)= \chi_1\big(f(x)\big)\chi_1\bigg(1+\frac
{f(x+tq_1)-f(x)}{f(x)}\bigg)\hat \chi\big(f(x+tq_1)\big).\ee Hence, as
in the proof of Theorem \ref{Theorem2},
\be\label{4.6.3}\chi\big(f(x+tq_1)\big)=
\chi_1\big(f(x)\big)\hat \chi\big(f(x+tq_1)\big) e_{q_1^m}
\Big(\sum^{m-1}_{j=1} Q_j (x) \;q_1^j\; t^j\Big), \ee where
\be\label{4.6.5} Q_j(x) =\frac 1{j!}\frac {d^j}{dt^j}\Big\{
F\Big(\frac {f(x+t)-f(x)}{f(x)}\Big) \Big\}\Big|_{t=0} \ee with
\be\label{4.6.6} F(x) =\sum_{s=1}^{2m} (-1)^{s-1} \frac 1s x^s
\qquad \text {(up to a factor)}. \ee

We estimate $\sum_{x\in I}\chi(f(x))$ by the same technique as used in the proof of Theorem \ref{Theorem3} with $d=m-1$ (see \eqref{3.2}-\eqref{3.8}). After averaging and summing over $t\in M$, and applying H\"older's inequality,
we remove the last factor in \eqref{4.6.3} and obtain
\be\label{12.29}
\frac 1N\left|\sum_{x\in I}\chi(f(x))\right|\ll\bigg[ \frac {M^{\;m^2}}{NM^{2k}}
 \sum_{0\leq t_1, \ldots, t_{2k}<M} \Big|\sum_{x\in I} \hat\chi
\big(R_{t_1, \ldots , t_{2k}}(x) \big)\Big|\bigg]^{\frac
1{2k}}
.\ee
Here $$R_{\underline{t}}(x):=R_{t_1, \ldots, t_{2k}} (x) =\frac {f(x+q_1t_1)\cdots f(x+q_1 t_k)}{f(x+ q_1 t_{k+1})\cdots f(x+q_1 t_{2k})}.
$$

 Choose $\underline{t}=(t_1,\cdots, t_{2k})\in [1, M]^{2k}$ such that
 $f_1(x)=R_{\underline{t}}(x)$ maximizes the inner character sum in the right-hand-side of \eqref{12.29} among all admissible $(R_{\underline{t}}, \bar q)$. Let $\mathcal B$ be  the set of $ \underline{t}$ such that $(R_{\underline{t}}, \bar q)$ is not admissible. Hence \eqref{12.29} gives
 \be\label{12.30}
 \frac 1N\Big|\sum_{x\in I}\chi(f(x))\Big|\ll M^{^{\frac 1{60}-1}}|\mathcal B|^{^{\frac 1{2k}}}+ M^{^{\frac 1{60}}} \Big|\frac 1N\sum_{x \in I} \hat \chi
\big(f_1(x)\big)\Big|^{\frac 1{60m^2}}.
\ee

 We want to give an upper bound on $|\mathcal B|$.

 \noindent{\it Claim.} $|\mathcal B| \ll M^{2k-\frac k6}.$.

 \noindent{\it Proof of Claim.} First, we observe that
 the zeros or poles of $R_{\underline{t}} (x)$ are of the form
\be\label{6.7}
b_{\alpha'}:=a_{\alpha} -t_{\nu}q_{1} \; \text{ with } t_{\nu}\in[1,M].
\ee

Second, we note that while applying H\"older's inequality to obtain \eqref{12.29}, we take $k\in \mathbb Z^+$  satisfying
\be\label{6.12} 48kd<\big(\log N\big)^\frac 1{10}\; \text{ and } k>30.\ee

To bound $|\mathcal B|$, we fix $p|\bar q$. In
 $R_{\underline{t}}(x) =\prod^{\beta'}_{\alpha'=1} (x-b_{\alpha'})^{d_{\alpha'}}$, we may assume $d_1=1$ and $a_1\not
=a_{\alpha}$ (mod $p$) for any $\alpha >1$. Recalling \eqref{6.7},
assume that none of the $a_1-t_{\nu}q_{1}, 1\leq \nu\leq 2k$, is
simple (mod $p$). This means that for each $\nu$ there is a pair
$(\alpha(\nu),\sigma    (\nu))$  in $\{1, \dots, \beta\} \times \{1,
\dots, 2k\}$ such that $\alpha(\nu)\neq 1,$ $\sigma(\nu)\neq\nu$ and
\be\label{5.8} a_{1}-t_{\nu}q_{1}\equiv
a_{\alpha(\nu)}-t_{\sigma(\nu)}q_{1}  \ (\mod p). \ee The important
point is that $\sigma(\nu)\neq \nu$ for all $\nu$, by assumption on
$a_{1}$. One may therefore obtain a subset $S \subset \{1,\dots,
2k\}$ with $| S | = k$ such that there exists $S_1\subset S$ with
$|S_1|=\frac k2$ and \be\label{5.9}
 S_{1} = \{ \nu\in S : \sigma (\nu)\notin S_1 \}.
\ee
(The existence of $S$ and $S_1$ satisfying this property is justified in Fact 6.1 following the proof of this lemma.)

Specifying the values of $t_{\nu'}$ for those $\nu' \in
\{1,\cdots, 2k\}\setminus
S_{1}$, equations \eqref{5.8} will determine the remaining
values, after
specification of $\alpha(\nu)$ and $\sigma (\nu)$.
An easy count shows that
\be\label{6.10}
\begin{aligned} &\big|\{\pi_p(\underline{t}): \;R_{\underline{t}}  \;\text{ is $p$-bad } \}\big|\\
\leq &\begin{pmatrix}{2k}\\{k} \end{pmatrix} \
\begin{pmatrix}{k}\\{\frac{k}{2}}\end{pmatrix}
\bigg(\frac 32 \;k\bigg)^{\frac{
k}{2}}\bigg(\frac{\beta}2\bigg)^{\frac{k}{2}}\big( \mathcal M\big)^{\frac{3k}{2}}< (48 kd)^{\frac k2} \big( \mathcal M\big)^{\frac{3k}{2}},
\end{aligned}
\ee
where $\mathcal M= \min(M, p)$.
The first factor counts the number of sets $S$, the second the number of sets $S_1$, and the third and the forth the numbers of maps $\sigma|_{S_1}$ and $ \alpha|_{S_1}$.

Applying assumptions (e)-(f) to \eqref{6.10}, we obtain
\be\label{6.11}\big|\{\pi_p(\underline{t}): \;R_{\underline{t}}  \;\text{ is $p$-bad } \}\big|
<\big( \mathcal M\big)^{\frac{8}{5}k}.\ee

If $(R_{\underline{t}} , \bar q)$ is not admissible, there is some $Q|\bar q$, $Q>q_r^{\tau}>\bar q^{\;\tau}$ such that for each $p|Q$, $R_{\underline{t}} $ is $p$-bad. As in the proof of Theorem \ref{Theorem3}, we distinguish several cases.

\smallskip

\noindent{\it $(a).$ There is $p|Q$ with $p>M$.}

Hence, $\big|\{\underline{t}\in [1, M]^{2k}: R_{\underline{t}}  \text{ is $p$-bad } \}\big|
< M^{\frac{8}{5}k}$ and summing over $p$ gives the contribution $M^{\frac{8}{5}k}\log\bar q$.
$$(b). \; \sqrt  M< \max_{p|Q} p <M.\qquad\qquad\qquad\qquad\qquad\qquad\qquad\qquad\qquad\qquad\qquad\qquad\qquad\qquad
\qquad\qquad\qquad\qquad\qquad\qquad\qquad\qquad\qquad\qquad\qquad\qquad
\qquad\qquad\qquad\qquad\qquad\qquad\qquad\qquad\qquad$$
Then
$$\begin{aligned}&\big|\{\underline{t}\in [1, M]^{2k}: R_{\underline{t}}  \text{ is $p$-bad } \}\big|\\
\leq &\Big(\frac Mp+1\Big)^{2k}\big|\{\pi_p(\underline{t}): R_{\underline{t}}  \text{ is $p$-bad } \}\big|\\
\leq & \Big(\frac Mp+1\Big)^{2k}p^{\frac 85k}<M^{2k}\Big(\frac p{32}\Big)^{-\frac 25 k}<(4M)^{\frac 95k}.\end{aligned}$$
Summing over $p$ gives the contribution $(4M)^{\frac 95k}\log \bar q$.
$$(c). \; \max_{p|Q} p\leq \sqrt M\text{ \it and } Q>M.\qquad\qquad\qquad\qquad\qquad\qquad\qquad\qquad\qquad\qquad\qquad\qquad\qquad\qquad
\qquad\qquad\qquad\qquad\qquad\qquad\qquad\qquad\qquad\qquad\qquad\qquad
\qquad\qquad\qquad\qquad\qquad\qquad\qquad\qquad\qquad$$
Take $Q_1|Q$ such that $\sqrt M< Q_1<M$. Then
$$\begin{aligned}&\big|\{\underline{t}\in [1, M]^{2k}: R_{\underline{t}}  \text{ is $p$-bad for each } p|Q_1 \}\big|\\
\leq &\Big(\frac M{Q_1}+1\Big)^{2k}\big|\{\pi_{Q_1}(\underline{t}): R_{\underline{t}}  \text{ is $p$-bad for each } p|Q_1 \}\big|\\
\leq &\Big(\frac M{Q_1}+1\Big)^{2k}\prod_{p|Q_1} p^{\frac 85k}<M^{2k} \Big(\frac {Q_1}{32}\Big)^{-\frac 25 k}<(4M)^{\frac 95k}.
\end{aligned}$$
Summing over $Q_1$ gives the contribution $(4M)^{\frac 95k +1}$.

\smallskip

Summing up cases (a)-(c) and recalling assumption (f), we conclude that
\be\label{6.12}\begin{aligned}|\mathcal B|=\;&\big|\{\underline{t}\in [1, M]^{2k}: (R_{\underline{t}} ,\bar q) \text{ is not admissible }\}\big|\\
< \;& M^{2k}\big(M^{-\frac k6} +\bar q^{\;\frac {-\tau k}5}\big)\ll M^{2k-\frac k6}. \end{aligned}\ee


 Putting \eqref{12.30} and \eqref{6.12} together, we obtain \eqref{6.4} and the lemma is proved. ${}\;\;\;\;\;\square$

\bigskip

\noindent{\bf Fact 6.1.} {\it Let $\mathcal K=\{1,\cdots, 2k\}$ and $\sigma : \mathcal K\to \mathcal K$
be a function such that $\sigma (\nu)\not=\nu$ for all $\nu\in \mathcal K$. Then there exist subsets
$S_1\subset S\subset \mathcal K$ with $|S_1|=\frac k2, \;|S|=k$ and $\sigma (\nu)\not\in S_1$ for any $\nu\in S$.}

\noindent{\it Proof.} Since the subset of elements of $\mathcal K$ with more than one pre-image of $\sigma$
has size $\leq k$, there exists a subset $S\subset \mathcal K$ with $|S|=k$ such that every $\nu\in S$ has at most one pre-image.
To construct $S_1\subset S$, we choose $\nu_i\in S$ inductively,
such that $\nu_i\not\in\{\nu_1, \ldots,\nu_{i-1}, \sigma(\nu_1), \ldots, \sigma( \nu_{i-1})\}\;\bigcup \;\sigma^{-1}\left(\{\nu_1
, \ldots, \nu_{i-1}\}\right)$ and $\sigma(\nu_i)\not\in  S_1.\quad\square$

\medskip

\noindent{\it Proof of Theorem 5.}

Following the claim in \S5, we will prove the theorem for $\chi=\chi_1\cdots \chi_r$, where $\chi_i (\mod q_i^{m_i})$
is arbitrary for $i<r$, and primitive for $i=r$, and $q=q_1^{m_1}\cdots q_{r-1}^{m_{r-1}}q_r$ satisfies (3) of the claim. We will apply lemma 6 repeatedly. First, we choose $M_{1}<M_{2}< \cdots < M_{r-1}$ a sequence of values of $M$, which will satisfy condition \eqref{26.18}. Then we will iterate \eqref{6.4}, losing a factor $\chi_i$  in $\chi$ and adding an error term $M_i                          ^{-1/12}$ in the bound each time.

  In order to satisfy Assumption (e), we assume \be\label{26.14}
1440\cdot 60^{r-1}\prod_{i=1}^{r}(m_i^2)<\big(\log N\big)^{\frac
1{5}},\ee which follows from \eqref{15.17} and \eqref{15.18}.

Let \be\label{26.22} M=e^{(\log q_r)^{9/10}},\ee and
for $ s=1, \cdots, r-1$, take $M_s,$ such that\be\label{26.18}
\begin{aligned}
M_1^{-\frac 1{15}}=& M^{-1}\\
M_{1}^{\frac 1{60}} M_2^{\frac 1{ 60^2m_1^2}}\cdots M_{s-1}^{\frac
1{60^{s-1}m_1^2\cdots m_{s-2}^2}}M_{s}^{-\frac 1{15\cdot
60^{s-1}m_1^2\cdots m_{s-1}^2}}=&M_1^{-\frac 1{15}}.
\end{aligned}\ee

One checks recursively that \be\label{26.19} M_s\leq M^{5^{s-1}\cdot
15^{s}m_1^2\cdots m_{s-1}^2}.\ee
Indeed, from \eqref{26.18},
$$M_{s-1}^{\;\;\;\frac 1{15\cdot
60^{s-2}m_1^2\cdots m_{s-2}^2}} =M_{s-1}^{\;\;\;\frac {-1}{
60^{s-1}m_1^2\cdots m_{s-2}^2}}M_{s}^{\frac 1{15\cdot
60^{s-1}m_1^2\cdots m_{s-1}^2}}.$$
Therefore, by induction
$$M_s=M_{s-1}^{75m_{s-1}^2}\leq M ^{5^{s-2}\cdot
15^{s-1}m_1^2\cdots m_{s-2}^2 (75m_{s-1}^2)}=M^{5^{s-1}\cdot
15^{s}m_1^2\cdots m_{s-1}^2}.$$

In order to satisfy the last condition in Assumption (f) of Lemma 6,
we assume \be\label{26.23} \sum_{i=1}^{r-1}\log
m_i <\frac 1{40} \log\log q_r,\ee
 (Clearly, this follows from \eqref{15.17}.) and note that
 \be\label{26.21} M^{(5\cdot 15)^{r-1}m_1^2\cdots
m_{r-2}^2}<q_r^{\tau}=q_r^{\frac {10}{\log\log q_r}}.\ee
(Since by \eqref{15.18}, $r < 10^{-3} \log \log q_r$.)

By \eqref{6.4} and  iteration, $\frac{1}{N} \big| \sum_{x=1}^N
\chi(x) \big|$ is bounded by\be\label{26.15}\begin{aligned}
M_{1}^{-\frac 1{15}} + M_{1}^{\frac 1{60}} M_2^{-\frac 1{15\cdot
60m_1^2}}&+M_{1}^{\frac 1{60}} M_2^{\frac 1{ 60^2m_1^2}}M_3^{-\frac
1{15\cdot 60^2m_1^2m_2^2}}+\cdots\\+M_{1}^{\frac 1{60}} M_2^{\frac
1{ 60^2m_1^2}}&\cdots M_{r-2}^{\frac 1{60^{r-2}m_1^2\cdots
m_{r-3}^2}}M_{r-1}^{-\frac 1{15\cdot 60^{r-2}m_1^2\cdots
m_{r-2}^2}}\\+&M_{1}^{\frac 1{60}} M_2^{\frac 1{ 60^2m_1^2}}\cdots
M_{r-1}^{\frac 1{60^{r-1}m_1^2\cdots m_{r-2}^2}}\mathcal S^{\frac 1{
 60^{r-1}m_1^2\cdots m_{r-1}^2}},\end{aligned}\ee where $\mathcal S$
is of the form \be\label{26.16} \mathcal S=\frac
1N\Big|\sum_{x=1}^N\chi_r(f(x))\Big|,\ee with $\chi_r$ primitive
modulo $q_r$, and
$$ f(x)=\prod _{\alpha =1}^{\beta} (x-a_{\alpha})^{d_{\alpha}},
a_{\alpha}, d_{\alpha}\in \mathbb{Z},$$ \be\label{26.17} d = \sum
\mid d_{\alpha} \mid < 60^{r}m^{2}_{1} \dots m^{2}_{r},\ee and $(f,
\bar q)$ admissible for some $\bar q|q_r$, $\bar q >\sqrt {q_r}$.

Condition \eqref{26.18} ensures that each of the $r-1$ fist terms
in \eqref{26.15} is bounded by $\frac 1M$.

Applying \eqref{26.19} to \eqref{26.15} gives
\be\label{26.20} \frac{1}{N} \big| \sum_{x=1}^N \chi(x) \big|
<\frac{r-1}M+M^{(\frac 54)^{r-1}-1}\mathcal S^{\frac
1{60^{r-1}m_1^2\cdots m_{r-1}^2}}.\ee

We will continue the proof of the theorem in the next section by distinguishing two cases.

\medskip

\section{The two cases.}

To finish the proof of Theorem \ref{Theorem8}, we need to bound
$\mathcal S$ in \eqref{26.20}. We will use Claim  in \S5. Recall \eqref{15.8} that
$$q_r=q_0^m, \;\; q_0 \text{ square-free }.$$
We will do induction on $m$.

\medskip

\noindent{\it Case 1. $m=1$.}

Since $(f, \bar q)$ is admissible and $\bar q$ is square-free, we may apply  Remark 3.3 to bound
$\mathcal S$. Therefore,\be\label{26.24} \mathcal
S < \bar q ^{\; -\frac 3{10}} q_r ^{\frac 3{10}\tau}<q_r^{-\frac
17},\ee and by \eqref{26.20} \be\label{26.25}\frac{1}{N} \big|
\sum_{x=1}^N \chi(x) \big| <\frac{r-1}M+M^{(\frac
54)^{r-1}-1}q_r^{-\frac 1{7\cdot 60^{r-1}m_1^2\cdots
m_{r-1}^2}}<\frac rM.\ee The last inequality is by \eqref{15.17} in Remark 5.2 and
\eqref{26.21}.

 Now we use \eqref{26.22} and  \eqref{15.18} to bound \eqref{26.25} and \eqref{15.8}-\eqref{15.6} to obtain \eqref{15.3}.
 $\quad\square$

\medskip

We state the above case as a proposition for its own interest.

 \noindent{\bf Proposition 7.}
{\it Assume $q = q_1^{m_1} \dots q_{r-1}^{m_{r-1}}q_r$  with
$(q_i,q_{j})=1$ for $i\not = j$, $q_r$ square-free and
\be\label{26.1} \prod_{i=1}^{r-1} m_i<\Big(\log q_r\Big)^{\frac
1{75}}.\ee Factor $ \chi=\chi_1\ldots \chi_r,$ where $\chi_i (\mod
q_i^{m_i})$ is arbitrary for $i<r$, and primitive for $i=r$.

 We further assume

{\rm(i)}. For all $ p | q_r, p > \sqrt{\log q_r}.$

{\rm(ii)}. For all $i$, $q_i^2< N<q $.

{\rm(iii)}. $r < 10^{-3} \log \log q_r$.

Then \be\label{26.2} \Big|\sum_{x\in I} \chi(x)\Big| <
N e^{-(\log q_r)^{4/5}},\ee
where $I$ is an interval
of size $N$.}

\medskip

\noindent{\it Case 2. $m>1$.}

In this situation, we follow the analysis in the proof of Lemma 6. (Particularly, see \eqref{4.6.4}-\eqref{4.6.6}.) To bound $\mathcal S$ in \eqref{26.15}, we will use Postnikov's theorem and Vinogradov's lemma ([Ga], Lemma 4) rather than Weil's estimate (Remark 3.3) as we did in Case 1.
Recall$$\mathcal S=\frac 1N\Big|\sum_{n=1}^N\chi_r(f(n))\Big|,$$
with $\chi_r$ primitive modulo $q_r$, and
 \be\label{1.12.2013}f(x)=\prod _{\alpha =1}^{\beta} (x-a_{\alpha})^{d_{\alpha}}\quad \text{with } d \leq 60^{r}m^{2}_{1} \dots m^{2}_{r},\ee
where $f(x)$ satisfies the property that for all $p|q_{0}$, $f(x)$ is $p$-good.

Write $n\in [1,N]$ as $n=x+tq_0$, with $ 1\leq x\leq q_0$ and $1\leq
t\leq \frac N{q_0}$. Then as in \eqref{4.6.4} and \eqref{4.6.5},
\be\label{4.15} \begin{aligned}N\cdot\mathcal S =\;&
\sum_{x=1}^{q_0}\sum^{N/q_{0}}_{t=1}
\chi_r\big(f(x)\big)e_{q^{m}_{0}}\Big (\sum^{m-1}_{j=1} Q_{j}(x)
q^{j}_{0} t^{j}\Big)\\\leq
&\sum_{x=1}^{q_0}\;\Big|\sum^{N/q_{0}}_{t=1} e_{q^{m}_{0}}\Big
(\sum^{m-1}_{j=1} Q_{j}(x) q^{j}_{0} t^{j}\Big)\Big|.
\end{aligned}\ee

 We want to find some information on the
coefficients $Q_j$ in \eqref {4.6.3}. We may assume in \eqref{1.12.2013} that $a_1 =0$ is a simple
zero or pole of $f$; replacing $f$ by $\frac 1f$ (which we may by
replacement of $\chi$ by $\bar{\chi}$), hence\be\label{4.6}
f(x) =xg(x)=x\prod_{a_{\alpha}\not=0}(x-a_{\alpha})^{d_{\alpha}} \qquad\mod p\ee
with $g(0)$ defined and non-vanishing $(\mod p)$.

From \eqref {4.6.5}, \eqref{4.6.6}, and \eqref{4.6}, we have \be\label{4.7} j! Q_j(x)
=\sum_s (-1)^{s-1} \frac 1{s(xg(x))^s} \frac {d^j}{dt^j}
\big[\big((x+t)g(x+t)-xg(x)\big)^s\big]\Big|_{t=0}. \ee Let $G(t)=(x+t)g(x+t)-xg(x)$. Since $G(t)$ divides $\frac {d^j}{dt^j}G(t)^s$ for $s>j$, and $G(0)=0$, in \eqref{4.7} only
the terms $s\leq j$ contribute and $Q_j$ has a pole at $0$ of order
$j$. Write \be\label {4.8}
 C\cdot Q_j(x) =\frac 1{x^j}+\frac {A_j(x)}{B_j(x)}
\ee with $A_j(x), B_j(x)\in\mathbb Z[x]$ and $B_j(x) =x^k\hat
B_j(x), k<j.$
Here \be\label{4.9} \hat
B_j(0)\not\equiv 0(\mod p), \ee since $B_j(x)$ is a product of monomials
of the form $x-a_\alpha$ and $a_\alpha\not\equiv 0 (\mod p)$ for $\alpha
\not= 1$. Thus \be\label{4.10} C\cdot Q_j(x) =\frac {P_j(x)}{x^j
\hat B_j(x)} \ee where $P_j(x) \in\mathbb Z[x]$ is of degree at most
$dj$, $P_j(0)\not\equiv 0 \; (\mod p)$. It follows that \be\label{4.11}
\big|\{1\leq x\leq p : Q_j(x) \equiv  0 \; (\mod p)\}\big|\leq  dj,
\ee
and
\be\label{4.12} \big|
\{1\leq x\leq q_0 : Q_{j}(x)\equiv 0 \; (\mod \bar q_0)\}\big| \leq
(dj)^{\omega(\bar q_0)} \frac {q_0}{\bar{q_0}}, \ee
whenever $\bar q_0 |q_0$.
Taking $j=m-1$ and fixing $x$, we will
apply Vinogradov's lemma (\cite{Ga}, Lemma 4) to bound the following inner double sum in \eqref{4.15}.
\be\label{1.16.2013} \Big|\sum^{N/q_{0}}_{t=1} e_{q^{m}_{0}}\Big
(\sum^{m-1}_{j=1} Q_{j}(x) q^{j}_{0} t^{j}\Big)\Big|. \ee

\noindent
{\bf Lemma} {\rm  (Vinogradov).}
{\it Let $f(t)=a_{1}t +\cdots + a_{k}t^{k} \in \mathbb{R}[t], k\geq2$ and $P\in\mathbb{Z}_{+}$ large.\\
Assume $a_{k}$ rational, $a_{k}=\frac{a}{b}, (a,b)=1$ such that

\be\label{4.13}
2< P\leq b\leq P^{k-1}
\ee
Then \be\label{4.14} \Big|\sum_{n\in I} e(f(n))\Big| < C^{k(\log
k)^{2}} P^{1-\frac{c}{k^{2}\log k}} \ee for any interval $I$ of
size $P$ $(c, C$ are constants). }

\medskip

Since the dominating saving in \eqref{4.14} is $P^{\frac{c}{k^{2}\log k}}$, we want the denominator of the leading coefficient $a_{m-1} =\frac {Q_{m-1}(x)}{q_0^m}q_0^{m-1}$ in \eqref{1.16.2013} big.

Let $\bar q_0=\big(Q_{m-1}(x), q_0\big)$. Write $Q_{m-1}(x) \equiv \bar q_0 \bar a\in\mathbb Z (\mod q_0)$.
 Therefore
$$a_{m-1} =\frac {Q_{m-1}(x)}{q_0}=\frac {\bar a}{\bar{\bar q}_0},\; \text{ with } \bar{\bar q}_0= \frac {q_0}{\bar q_0}\;\text{ and } (\bar a, \bar{\bar q}_0)=1.$$

To sum $x\in [1, q_0]$ in \eqref{4.15}, we distinguish the cases according to $\big(Q_{m-1}(x), q_0\big)$.

\noindent{\it Case} (i). $\bar q_0=\bar q_0(x)=(Q_{m-1}(x), q_0)\leq\sqrt{q_0}$. Hence $\bar{\bar q}_0 >\sqrt{q_0}$.

Divide the interval $[1, N/q_0]$ into subintervals of length $\sqrt {q_0}$ each. Applying \eqref{4.14} with $P=\sqrt{q_0}$ to the subsum over each subinterval and summing up the subsums give \be\label{4.16}\begin{aligned}
\Big|\sum^{N/q_{0}}_{t=1} e_{q^{m}_{0}}\Big (\sum^{m-1}_{j=1}
Q_{j}(x) q^{j}_{0} t^{j}\Big)\Big|&< \frac N{q_0} \ C^{m(\log m)^2} q_0^{-\frac
c{2m^2(\log m)}}
. \end{aligned}\ee
The last inequality follows from \eqref{15.8}-\eqref{15.6}.

\noindent{\it Case} (ii). $\bar q_0=\bar q_0(x)=(Q_{m-1}(x), q_0)>\sqrt{q_0}$. We will use the trivial bound $N/q_0$ on \eqref{1.16.2013}. It remains to estimate the number of
those $1\leq x\leq q_0$ such that $Q_{m-1}(x)\equiv 0$ in $\mathbb
Z/\bar q_0\mathbb Z$ for some $\bar q_0 >\sqrt{q_0}$. This number is
by \eqref{4.12} at most \be\label{6.6} \sum_{\substack{\bar
q_0|q_0\\ \bar q_0>\sqrt{q_0}}} (dm)^{\omega (\bar q_0)}
\frac{q_0}{\bar q_0} < 2^{\omega(q_0)} (dm)^{\omega(q_0)}\sqrt{q_0}<
(2dm)^{\frac {2\log q_0}{\log\log N}} \sqrt{q_0}, \ee since all prime
divisors of $q_{0}$ are at least $(\log N)^{\frac12}$. Note that the
degree $d$ of $f(x)$ is bounded by \eqref{26.17}. Applying
\eqref{15.6} and \eqref{15.17}, we have \be\label{26.7} dm<60^r
\big(\log q_r\big)^{\frac 2{75}} \big(\log
N\big)^{3c}<\big(\log N\big)^{15c}.\ee

In particular, \eqref{26.7} will ensure that \eqref{6.6} is bounded
by $ q_0^{3/4}$.

Applying Case (i) and Case (ii) to \eqref{4.15}, we have \be\label{6.8}
\begin{aligned}
&\sum_{x=1}^{q_0}\Big|\sum^{N/q_{0}}_{t=1} e_{q^{m}_{0}}\Big (\sum^{m-1}_{j=1}
Q_{j}(x) q^{j}_{0} t^{j}\Big)\Big|\\= &\sum_{\bar q_0(x)\leq\sqrt{q_0}}\Big|\sum^{N/q_{0}}_{t=1} e_{q^{m}_{0}} \Big(\sum^{m-1}_{j=1}
Q_{j}(x) q^{j}_{0} t^{j}\Big)\Big|+\sum_{\bar q_0(x)>\sqrt{q_0}}\Big|\sum^{N/q_{0}}_{t=1} e_{q^{m}_{0}} \Big(\sum^{m-1}_{j=1}
Q_{j}(x) q^{j}_{0} t^{j}\Big)\Big|\\
\leq&\;\;q_0\;\frac N{q_0} C^{m(\log m)^2} q_0^{-\frac
c{2m^2(\log m)}}+ q_0^{\frac 34}\;\frac N{q_0}\\<&\;\;N C^{m\log m)^2} q_0^{-\frac
c{2m^2(\log m)}}. \end{aligned}\ee

Therefore \be\label{7.21.1}\mathcal S<C^{m\log m)^2} q_0^{-\frac
c{2m^2(\log m)}}.\ee

Now we want to apply \eqref{7.21.1} to \eqref{26.20}. In \eqref{26.20}, using \eqref{15.18}, we have a bound $(\log q_r)^\epsilon$ on the exponent of $M$ in the second term in \eqref{26.20}. On the other hand, after applying  \eqref{7.21.1} in the second factor of the second term in \eqref{26.20}, and using \eqref{15.17}, \eqref{15.6} and \eqref{15.5}, we bound the exponent of $q_0$ by $-(\log q_r)^{-\epsilon}$ as well. Hence we have
\be\label{00}\frac 1N\Big|\sum^{N}_{x=1} \chi
(x)\Big| < \frac{r-1}M+e^{(\log q_r)^{9/10+\epsilon}}\;q_0^{-(\log q_r)^{-\epsilon}}\ee

By \eqref{15.8}-\eqref{15.6}, the factor $\;q_0^{-(\log q_r)^{-\epsilon}}$ in \eqref{00} is bounded by $\exp(-(\log q_r)^{1-\epsilon'}) $. Hence \eqref{00} is bounded by $e^{-(\log q_r)^{9/10}}$ after applying \eqref{15.18} and \eqref{26.22}.
This proves the theorem.$\quad\square$

\section{Applications.}


Repeating the argument in deducing Theorem 4 from Theorem 3' and Theorem 3'', we obtain the following mixed character sum estimate from the proof of Theorem 5.

\medskip

There is the following mixed character sum version of Theorem 5.

\medskip

\noindent {\bf Theorem 8.} {\it Under the assumptions of Theorem
\ref{Theorem8}, \be\label {7.1} \Big|\sum_{x\in I}\chi(x) e^{if(x)}\Big|< N
e^{-\sqrt{\log N}} \ee assuming $f(x)\in\mathbb R[x]$ of degree at most
$(\log N)^c$ for some $c>0$.}

\medskip

A more precise statement is again possible, but the above one is all we need for what follows.
To prove Theorem 8, simply go back to the opening argument in Theorem 4 which removes the factor
$e^{if(x)}$ at the cost of replacing $\chi(x)$ by $\chi(R(x))$ with $R(x)$ a certain rational function of $x$.
Note that this step is already part of the proof of Lemma 6. At this point, proceed further with \S6
and \S 7 as in proving Theorem 5.

\medskip

\noindent{\bf Corollary 9.} {\it Assume $N$ satisfies
$$q>N>\max_{p|q} p^{10^3}$$
and $q$ satisfies
\be\label{7.6} \log N
>(\log qT)^{1-c}+ C\log \Big( 2\frac {\log q}{\log q'}\Big)\;\frac
{\log q'}{\log\log q}\;. \ee
Then for $\chi$ primitive, we have
\be\label{7.2} \Big|\sum_{n\in I}\chi(n) n^{it}\Big| < N
e^{-\sqrt{\log N}}. \ee }

\medskip

From Corollary 9, one derives bounds on the Dirichlet L-function $L(s, \chi)$ and zero-free
regions the usual way. See for instance Lemmas 8-11 in $[\;I\;]$. This leads to the following theorem.

\medskip

\noindent{\bf Theorem 10.} {\it Let $\chi$ be a primitive multiplicative character with modulus $q$,
$\mathcal P=\max _{p|q}p,$ $ q'=\prod_{p|q}p,$ and $ K=\frac{\log q}{\log q'}.$
For $T>0$, let $$\theta=c\min\Big(\frac 1{\log
\mathcal P}, \frac {\log\log q'}{(\log q')\log 2K},\frac 1{(\log
qT)^{1-c'}}\Big).$$ Then the Dirichlet L-function
$L(s,\chi)=\sum_n\chi(n) n^{-s}, s=\rho+it$ has no zeros in the region
$\rho>1-\theta, \; |t|<T$, except for possible Siegel zeros.}

\medskip

It follows in particular that $\theta\log qT\to\infty$ if $\frac {\log \mathcal P}{\log q}\to 0$.

From Theorem 10, we have the following.

\medskip

\noindent
{\bf Corollary 11.}
{\sl Assume $q$ satisfies that $\log p=o(\log q)$ for any $p|q$. If $(a, q)=1$, then there is a prime $P\equiv a(\mod q)$ such that
$P< q^{\frac {12} 5+o(1)}$.}

\medskip

To deduce Corollary  11,
we follow the exposition of Linnik's theorem in [IK] (see p.440). Define $$\psi(x, q,a)=\frac 1{\phi(q)}\sum_{\chi \mod q} \bar\chi(a)\sum_{n\leq x}\Lambda(n)\chi(n),$$
where $\Lambda$ is the von Mangoldt function. Assuming $x>q^{\frac {12}5+\epsilon}$, we have
\be\label{2.8.3} \begin{aligned}&\psi(x, q,a)\\=\;&\frac x{\phi(q)}\left\{ 1-\sigma \frac{x^{\beta_1-1}}{\beta_1}+O\left(\frac {x^{-c\epsilon}}{\epsilon}\right)+O\left(\frac {x^{-c\eta}}{\epsilon}\right)
+O\left(\frac {\log q}q\right)\right\}\end{aligned}\ee
with
\be\label{2.8.4} \eta >c \min\left(\frac 1{\log \mathcal P}, \frac {\log\log q}{\log q}\right),\;\;\mathcal P=\max_{p|q} p.\ee
Here, on the right hand side of \eqref{2.8.3}, the second term accounts for a possible Siegel zero $s=\beta_1$ in  Theorem 10
(in which case $\sigma=1$, otherwise $\sigma=0$), while the third term is from Huxley's density estimate [H], and the fourth term from the zero-free region given by \eqref{2.8.4}. Certainly, $x^{-c\eta}\to 0$, if $\frac{\log\mathcal P}{\log q}\to 0$ (with $\epsilon$ fixed).
Also,
$$1-\frac {x^{\beta_1-1}}{\beta_1}\geq \beta_1-q^{-(1-\beta_1)}=1-\frac \gamma{\log q}-e^{-\gamma}$$
with $\gamma=(1-\beta_1)\log q$.

We distinguish two cases. If $\gamma$ is sufficiently small, then Corollary 2 in [HB2]  applies.
In this case there is a prime $ P\equiv a\pmod q$ with $P<q^{2+\delta}<q^{\frac {12}5}$.
 Otherwise, the right hand side of  \eqref{2.8.3} is greater than zero. Hence the conclusion in Corollary 11 holds in either case.

Next, following Goldmakher's work [G], we will derive from Theorem 5 the following theorem.

\medskip

\noindent{\bf Theorem 12.} {\it Let $\chi$ be a primitive multiplicative character with modulus $q$,
$\mathcal P=\max_{p|q} p$, $  q'=\prod_{p|q}p$ and
$K=\frac{\log q}{\log q'}$.\\  Let $M=(\log q)^{1-c}+\frac {\log q'}{\log\log q'}\log 2K+\log \mathcal P. $ Then
$$\Big|\sum_{n<x}\chi(n)\Big|\ll\sqrt q \sqrt{\log q}\;\sqrt M\sqrt{\log\log\log q}.$$}

We will sketch the argument and refer to [G] for more details.
  Denoting
$$\mathcal S_{\chi}(x)=\sum_{n<x}\chi(n)$$
with $\chi$ primitive modulo $q$. Proposition  2.2 in [G] states that
\be\label{2.8.5}
|\mathcal S_{\chi}(x)|\ll \sqrt q\;(\log q)^{\frac 12}\;\left|L(1+\frac 1{\log q}, \chi \bar\xi)\right|^{\frac 12}+\sqrt q\;(\log q)^{\frac 67}\ee
for some primitive character $\xi\pmod m$ of conductor less than $(\log q)^{1/3}$. This result uses crucially the work of Granville and Soundararajan [GS]. Let $\psi\pmod Q$ be the primitive character which induces $\chi\xi$. Then, by [G], Lemma 5.1 and Lemma 5.2,
\be\label{2.8.6} \frac qm\leq Q \leq qm\ee
and
\be\label{2.8.7}\left|\frac{L(s, \chi\bar\xi)}{L(s, \psi)}\right|\ll 1+\log\log m.\ee
Taking $s=1+\frac 1{\log q}$, we get by partial summation that
\be\label{2.8.8}L(s, \psi)=s\int_1^\infty \frac 1{t^{s+1}}\left(\sum_{n\leq t}\psi(n)\right) dt.\ee
For $t>Q$, estimate $|\sum_{n\leq t}\psi(n)|$ trivially by $Q$, which contributes in \eqref{2.8.8} for $O(1)$.
Next,
\be\label{2.8.9}\int_1^Q \frac 1{t^{s+1}}\left|\sum_{n\leq t}\psi(n)\right| dt \ll \log T+
\int_T^Q \frac 1{t^{2}}\left|\sum_{n\leq t}\psi(n)\right| dt\ee
and we apply Theorem 5 to bound $\sum_{n\leq t}\psi(n)$ in the second term of the
right hand side of \eqref{2.8.9}. Note that by \eqref{2.8.6} the expression \eqref{2.8.11} in Theorem 5
is essentially preserved, if $Q$ is replaced by $q$. Thus $T=N$ has to be chosen to satisfy \eqref{2.8.11} and Theorem 12 follows from \eqref{2.8.5}, \eqref{2.8.7}, and \eqref{2.8.9}.


\bigskip

\noindent{\it Acknowledgement.}  The author would like to thank the referee for very careful reading of the paper, which greatly improved its presentation. The author would also like to thank I. Shparlinski for helpful comments and the mathematics department of University of California at Berkeley for hospitality. Part of the work was finished when the author was in residence at the Mathematical Science Research Institute and supported by the NSF Grant~0932078000.

\end{document}